\documentclass[a4paper,11pt]{amsart}

\usepackage{amssymb}
\usepackage{mathrsfs}
\usepackage{cite}
\usepackage{url}

\newtheorem{theorem}{Theorem}[section]
\newtheorem{lemma}[theorem]{Lemma}
\newtheorem{corollary}[theorem]{Corollary}
\newtheorem{proposition}[theorem]{Proposition}

\newtheorem{xtheorem}{Theorem}
\newtheorem*{xconjecture}{Conjecture}

\theoremstyle{definition}
\newtheorem{remark}[theorem]{Remark}

\numberwithin{equation}{section}

\makeatletter
\@namedef{subjclassname@2020}{%
  \textup{2020} Mathematics Subject Classification}
\makeatother

\DeclareMathOperator{\RE}{Re}
\DeclareMathOperator{\IM}{Im}
\DeclareMathOperator{\meas}{meas}
\DeclareMathOperator{\ord}{ord}

\DeclareMathOperator{\supp}{supp}
\DeclareMathOperator{\sgn}{sgn}

\begin{document}

\title[New developments toward the Gonek Conjecture]
{New developments toward the Gonek Conjecture on the Hurwitz zeta-function}

\author[M. Mine]{Masahiro Mine}
\address{Global Education Center\\ Waseda University\\
1-6-1 NishiWaseda, Shinjuku-ku, Tokyo 169-8050, Japan}
\email{m-mine@aoni.waseda.jp}

\date{}

\begin{abstract}
In this paper, we prove a version of the universality theorem for the Hurwitz zeta-function in the case where the parameter is algebraic and irrational. 
Then we apply the result to show that many of such Hurwitz zeta-functions have infinitely many zeros in the right half of the critical strip. 
\end{abstract}



\maketitle

\section{Introduction and statements of results}\label{sec:1}
Let $s=\sigma+it$ be a complex variable. 
Denote by $\zeta(s,\alpha)$ the Hurwitz zeta-function with a parameter $0<\alpha \leq1$. 
It is defined by the Dirichlet series
\begin{gather}\label{eq:04010001}
\zeta(s,\alpha)
= \sum_{n=0}^{\infty} (n+\alpha)^{-s}
\end{gather}
on the half-plane $\sigma>1$ and can be continued meromorphically to the whole complex plane $\mathbb{C}$. 
Throughout this paper, let $D$ and $\mathcal{A}$ denote
\begin{gather*}
D
= \{s \in \mathbb{C} \mid 1/2< \sigma <1 \}, \\
\mathcal{A}
= \{0<\alpha<1 \mid \text{$\alpha$ is algebraic and irrational} \}, 
\end{gather*}
and $\meas \{\cdot\}$ stands for the usual Lebesgue measure of a measurable set $\{\cdot\}$ in $\mathbb{R}$. 
Gonek conjectured in his thesis \cite[p.~122]{Gonek1979} that $\zeta(s,\alpha)$ has a universality property for any $\alpha \in \mathcal{A}$ in the following sense. 

\begin{xconjecture}[Gonek]
Let $\alpha \in \mathcal{A}$. 
Let $K$ be a compact subset of the strip $D$ with connected complement. 
Let $f$ be a continuous function on $K$ which is analytic in the interior of $K$. 
Then, for every $\epsilon>0$, we have 
\begin{gather*}
\liminf_{T \to\infty} \frac{1}{T} \meas
\left\{ \tau \in [0,T] ~\middle|~ \sup_{s \in K} |\zeta(s+i \tau, \alpha)-f(s)|<\epsilon \right\}
> 0. 
\end{gather*}
\end{xconjecture}

The history of universality is briefly summarized in Section \ref{sec:1.1}. 
Here, we recall that the universality theorem of $\zeta(s,\alpha)$ has already been proved in the case where $\alpha$ is transcendental or rational. 
Furthermore, in the case $\alpha \in \mathcal{A}$, some progress was made by Sourmelidis and Steuding \cite{SourmelidisSteuding2022}. 
They succeeded to show an effective but weak form of universality of $\zeta(s,\alpha)$ in the same manner as in \cite{GarunkstisLaurincikasMatsumotoSteudingSteuding2010}. 
See Theorem \ref{thm:1.3} for the details. 
The main result of this paper also presents a version of universality of $\zeta(s,\alpha)$ for any $\alpha \in \mathcal{A}$, which is quite different from \cite{SourmelidisSteuding2022}. 
It provides another piece of evidence for Gonek's conjecture. 

\begin{xtheorem}\label{thm:1}
Let $\alpha \in \mathcal{A}$. 
Let $K$ be a compact subset of the strip $D$ with connected complement. 
Let $f$ be a continuous function on $K$ which is analytic in the interior of $K$. 
Then there exists a sequence $\{\alpha_k\}$ of elements in $\mathcal{A}$ depending on $\alpha,f,K$ with the following property: 
for every $\epsilon>0$, there exists a number $k_0(\epsilon)$ such that
\begin{gather*}
\liminf_{T \to\infty} \frac{1}{T} \meas
\left\{ \tau \in [0,T] ~\middle|~ \sup_{s \in K} |\zeta(s+i \tau, \alpha_k)-f(s)|<\epsilon \right\}
> 0
\end{gather*}
and $|\alpha_k-\alpha|<\epsilon$ for all $k \geq k_0(\epsilon)$. 
\end{xtheorem}

Notice that the sequence $\{\alpha_k\}$ of Theorem \ref{thm:1} converges to $\alpha$ as $k \to\infty$. 
If one can take $\alpha_k=\alpha$ for any $k$, then Gonek's conjecture is true. 

Theorem \ref{thm:1} yields a new result on the distribution of zeros of $\zeta(s,\alpha)$. 
The study of zeros of the Hurwitz zeta-function has a long history since the classical work of Davenport and Heilbronn \cite{DavenportHeilbronn1936a}. 
They proved that $\zeta(s,\alpha)$ has infinitely many zeros in the half-plane $\sigma>1$ if the parameter $\alpha$ is transcendental or rational with $\alpha \neq 1,1/2$. 
Then, Cassels \cite{Cassels1961} extended the result to the case $\alpha \in \mathcal{A}$. 
(Actually, his original proof contains an error, and it is corrected in \cite{Mine2024+}.) 
Studying the zeros in the strip $D$ is more difficult since \eqref{eq:04010001} is not available. 
To this day, we know that $\zeta(s,\alpha)$ has infinitely many zeros in $D$ only if $\alpha$ is transcendental or rational with $\alpha \neq 1,1/2$. 
Garunk\v{s}tis \cite{Garunkstis2005} showed that there exist infinitely many $\alpha \in \mathcal{A}$ such that the Hurwitz zeta-function $\zeta(s,\alpha)$ has an arbitrary finite number of zeros in $D$. 
In this paper, we refine the result so that $\zeta(s,\alpha)$ has infinitely many zeros in $D$. 

\begin{xtheorem}\label{thm:2}
There exist infinitely many $\alpha \in \mathcal{A}$ such that the Hurwitz zeta-function $\zeta(s,\alpha)$ has infinitely many zeros in the strip $D$. 
\end{xtheorem}

In fact, Theorems \ref{thm:1} and \ref{thm:2} are proved in a stronger form. 
For example, we obtain a lower bound for the number of the zeros of $\zeta(s,\alpha)$ in the rectangle $\sigma_1 \leq \sigma \leq \sigma_2$, $0 \leq t \leq T$ for any $1/2<\sigma_1<\sigma_2<1$. 
We postpone the exact statements, which appear in Section \ref{sec:1.2}.

\subsection{Universality theorems}\label{sec:1.1}
The notion of universality in analysis first appeared in the classical result of Fekete reported in \cite{Pal1915}. 
He proved that there exists a real formal power series $\sum_{n=1}^{\infty} a_n x^n$ with the following property: for every continuous function $f$ on $[-1,1]$ satisfying $f(0)=0$, there exists an increasing sequence $\{m_k\}$ of positive integers such that
\begin{gather*}
\sup_{x \in [-1,1]} \left|\sum_{n=1}^{m_k} a_n x^n - f(x)\right|
\to 0
\end{gather*}
as $k \to\infty$. 
Some similar approximation results were also obtained in \cite{Birkhoff1929, Marcinkiewicz1935}. 
Then such a kind of property was named universality by Marcinkiewicz \cite{Marcinkiewicz1935}. 
While these results in \cite{Pal1915, Birkhoff1929, Marcinkiewicz1935} do not provide explicit examples of an object with universality, Voronin \cite{Voronin1975b} achieved the universality theorem for the Riemann zeta-function $\zeta(s)$ as follows. 

\begin{theorem}[Voronin]\label{thm:1.1}
Let $0<r<1/4$. 
Let $f$ be a non-vanishing continuous function on the disc $|s| \leq r$ which is analytic in the interior. 
Then, for every $\epsilon>0$, there exists a positive real number $\tau=\tau(\epsilon)$ such that 
\begin{gather*}
\sup_{|s| \leq r} \left|\zeta \bigg(s+\frac{3}{4}+i \tau \bigg)-f(s) \right|
< \epsilon. 
\end{gather*}
\end{theorem}

There are several methods for the proof of the universality of $\zeta(s)$. 
See Bagchi \cite{Bagchi1981}, Gonek \cite{Gonek1979}, and Good \cite{Good1980}. 
The Linnik--Ibragimov conjecture vaguely asserts that any reasonable Dirichlet series would have a universality property. 
Then the universality of $\zeta(s)$ has been extended to many zeta and $L$-functions: Dirichlet $L$-functions \cite{Voronin1975b}; Dedekind zeta-functions \cite{Reich1980}; $L$-functions attached to certain cusp forms \cite{LaurincikasMatsumoto2001}, and so on. 
For more information, see Matsumoto's survey \cite{Matsumoto2015}. 
The universality of the Hurwitz zeta-function was proved by Bagchi \cite{Bagchi1981} and Gonek \cite{Gonek1979} independently, where the parameter is transcendental or rational. 

\begin{theorem}[Bagchi, Gonek]\label{thm:1.2}
Let $0<\alpha \leq1$ be a transcendental or rational number with $\alpha \neq 1,1/2$. 
Let $K$ be a compact subset of the strip $D$ with connected complement. 
Let $f$ be a continuous function on $K$ which is analytic in the interior of $K$. 
Then, for every $\epsilon>0$, we have
\begin{gather*}
\liminf_{T \to\infty} \frac{1}{T} \meas
\left\{ \tau \in [0,T] ~\middle|~ \sup_{s \in K} |\zeta(s+i \tau, \alpha)-f(s)|<\epsilon \right\}
> 0. 
\end{gather*}
\end{theorem}

If $\alpha$ is transcendental, then the real numbers $\log(n_1+\alpha), \ldots, \log(n_k+\alpha)$ are linearly independent over $\mathbb{Q}$ for any distinct integers $n_1,\ldots,n_k \geq0$. 
Therefore the Kronecker--Weyl theorem is available to show that $((n_1+\alpha)^{i \tau}, \ldots, (n_k+\alpha)^{i \tau})$ is uniformly distributed for $\tau \in \mathbb{R}$. 
This is a key step in the proof of Theorem \ref{thm:1.2} in the transcendental case. 
If $\alpha$ is rational, then we have the formula
\begin{gather}\label{eq:04010002}
\zeta(s,\alpha)
= \frac{q^s}{\phi(q)} \sum_{\chi \bmod q} \overline{\chi}(a) L(s,\chi)
\end{gather}
for $\alpha=a/q$ with $a,q \in \mathbb{Z}_{>0}$, where $\phi(q)$ denotes Euler's totient function, and $\chi$ runs through the Dirichlet characters of modulo $q$. 
Then the hybrid joint universality of Dirichlet $L$-functions implies Theorem \ref{thm:1.2} in the rational case. 

If $\alpha \in \mathcal{A}$, then the linear independence of $\log(n_1+\alpha), \ldots, \log(n_k+\alpha)$ is not necessary valid, and no formula similar to \eqref{eq:04010002} is known. 
These difficulties prevent us from resolving the universality of $\zeta(s,\alpha)$ in this case. 
The result of Sourmelidis and Steuding \cite{SourmelidisSteuding2022} provided the first progress on this problem. 
Let $0<c<1$ and $d \geq1$. 
Then we define the set $\mathcal{A}(c,d)$ as
\begin{gather*}
\mathcal{A}(c,d)
= \{\alpha \in \mathcal{A} \mid \text{$c \leq \alpha \leq 1$ and $\deg(\alpha) \leq d$}\}, 
\end{gather*}
where $\deg(\alpha)$ denotes the degree of $\alpha$. 
Let $\eta=4.45$ and $\theta=4/(27 \eta^2) \approx 0.00748$. 
We determine a positive real number $\iota$ by the equation $\iota^3-2\iota^2=2\theta$. 
Then we have $\iota \approx 2.003$. 
Furthermore, we put $\xi=\theta/\iota^2 \approx 0.00186$. 

\begin{theorem}[Sourmelidis--Steuding]\label{thm:1.3}
Let $0<\nu<\xi$, $1-\xi+\nu \leq \sigma_0 \leq 1$, and 
\begin{gather*}
\mu
\leq \left(\frac{\theta}{1-\sigma_0+\nu}\right)^{1/2}. 
\end{gather*} 
Let $f$ be a continuous function on $\mathcal{K}=\{s \in \mathbb{C} \mid |s-s_0| \leq r\}$ which is analytic in the interior of $\mathcal{K}$, where $s_0=\sigma_0+it_0$ with $t_0 \in \mathbb{R}$ and $r>0$. 
Let $0<c<1$ and $d=\mu^2-\theta/\mu^2+\nu$. 
Then, for every $\epsilon \in (0,|f(s_0)|)$, there exists a finite subset $\mathcal{E} \subset \mathcal{A}(c,d)$ with the following property: 
for any $\alpha \in \mathcal{A}(c,d) \setminus \mathcal{E}$, there exist real numbers $\tau \in [T,2T]$ and $\delta=\delta(\epsilon,f,T)>0$ such that
\begin{gather*}
\max_{|s-s_0| \leq \delta r} |\zeta(s+i \tau, \alpha)-f(s)|
< \epsilon, 
\end{gather*}
where $T=T(\epsilon,f,\alpha)$ is a large real number which can be effectively computable. 
The subset $\mathcal{E}$ is described by several effective constants depending on $\epsilon$ and $f$. 
Lastly, the real number $\delta$ is also effectively computable by choosing it to satisfy 
\begin{gather*}
\max_{|s-s_0|=r} |\zeta(s+i \tau, \alpha)| \frac{\delta^N}{1-\delta}
\leq \frac{1}{3}(2-e^{\delta r}) \epsilon 
\end{gather*}
for sufficiently large $N$. 
\end{theorem}

In fact, we can drop the dependence of $\delta$ on $T$ if $r$ satisfies $0<r<\sigma_0-1/2$. 
This was suggested by Y.~Lee, and a sketch of his idea was described at the end of Section 4 of \cite{SourmelidisSteuding2022}. 
It seems impossible to drop the remaining dependence of $\delta$ by their method.

\subsection{Statements of results}\label{sec:1.2}
Let $0<c<1$. 
Then we define the set $\mathcal{A}_\rho(c)$ as
\begin{gather*}
\mathcal{A}_\rho(c)
= \{\alpha \in \mathcal{A} \mid |\alpha-c| \leq \rho\}, 
\end{gather*}
where $\rho$ satisfies $0<\rho<\min\{c, 1-c\}$. 
In this paper, we prove the universality of $\zeta(s,\alpha)$ for all but finitely many $\alpha \in \mathcal{A}_\rho(c)$. 

\begin{theorem}\label{thm:1.4}
Let $0<c<1$. 
Let $K$ be a compact subset of the strip $D$ with connected complement. 
Let $f$ be a continuous function on $K$ which is analytic in the interior of $K$. 
Then, for every $\epsilon>0$, there exist a positive real number $\rho$ and a finite subset $\mathcal{E} \subset \mathcal{A}$ such that 
\begin{gather*}
\liminf_{T \to\infty} \frac{1}{T} \meas
\left\{ \tau \in [0,T] ~\middle|~ \sup_{s \in K} |\zeta(s+i \tau, \alpha)-f(s)|<\epsilon \right\}
> 0
\end{gather*}
for any $\alpha \in \mathcal{A}_\rho(c) \setminus \mathcal{E}$, where $\rho$ and $\mathcal{E}$ depend on $c,\epsilon,f,K$.  
\end{theorem}

This is not an effective result such as Theorem \ref{thm:1.3}, but we can extend a type of universality of $\zeta(s,\alpha)$ to functions on any compact subset $K$ of $D$ with connected complement. 
It is worth noting that $K$ is independent to $\epsilon,f,T$ in Theorem \ref{thm:1.4}, which is a major improvement from the previous work. 
The proof of Theorem \ref{thm:1.3} in \cite{SourmelidisSteuding2022} was derived by incorporating the ideas of Good \cite{Good1980} and Voronin \cite{Voronin1988b}. 
Then it was necessary to show a joint denseness theorem for the values of derivatives of the Hurwitz zeta-function. 
See \cite[Theorem 1]{SourmelidisSteuding2022}. 
To prove Theorem \ref{thm:1.4}, we do not need such a result. 
Instead, we deduce from Theorem \ref{thm:1.4} the following result. 

\begin{theorem}\label{thm:1.5}
Let $0<c<1$. 
Let $1/2<\sigma_0<1$ and $\underline{z}=(z_0, \ldots, z_N) \in \mathbb{C}^{N+1}$. 
Then, for every $\epsilon>0$, there exist a positive real number $\rho$ and a finite subset $\mathcal{E} \subset \mathcal{A}$ such that 
\begin{gather*}
\liminf_{T \to\infty} \frac{1}{T} \meas
\left\{ \tau \in [0,T] ~\middle|~ \max_{0 \leq n \leq N} |\zeta^{(n)}(\sigma_0+i \tau, \alpha)-z_n|<\epsilon \right\}
> 0
\end{gather*}
for any $\alpha \in \mathcal{A}_\rho(c) \setminus \mathcal{E}$, where $\rho$ and $\mathcal{E}$ depend on $c,\epsilon,\sigma_0,\underline{z}$.  
\end{theorem}

Another consequence of Theorem \ref{thm:1.4} concerns the estimate for the number of zeros of $\zeta(s,\alpha)$. 
Denote by $N_\alpha(\sigma_1,\sigma_2,T)$ the number of the zeros of $\zeta(s,\alpha)$ in the rectangle $\sigma_1 \leq \sigma \leq \sigma_2$, $0 \leq t \leq T$ counted with multiplicity. 
Then an upper bound for $N_\alpha(\sigma_1,\sigma_2,T)$ is well-known for any $0<\alpha \leq1$. 
Indeed, we apply Littlewood's lemma \cite[p.~220]{Titchmarsh1986} to obtain 
\begin{gather*}
N_\alpha(\sigma_1,\sigma_2,T)
\ll T
\end{gather*}
as $T \to\infty$ for any $1/2<\sigma_1<\sigma_2<1$. 
On the other hand, we have a lower bound for $N_\alpha(\sigma_1,\sigma_2,T)$ of the same magnitude for all but finitely many $\alpha \in \mathcal{A}_\rho(c)$. 

\begin{theorem}\label{thm:1.6}
Let $0<c<1$ and $1/2<\sigma_1<\sigma_2<1$. 
Then there exist a positive real number $\rho$ and a finite subset $\mathcal{E} \subset \mathcal{A}$ such that 
\begin{gather*}
N_\alpha(\sigma_1,\sigma_2,T)
\gg T 
\end{gather*}
as $T \to\infty$ for any $\alpha \in \mathcal{A}_\rho(c) \setminus \mathcal{E}$, where $\rho$ and $\mathcal{E}$ depend on $c,\sigma_1,\sigma_2$.  
\end{theorem}

Here, we check that the above results imply Theorems \ref{thm:1} and \ref{thm:2} described before. 
Let $\alpha, f, K$ be as in the statement of Theorem \ref{thm:1}. 
Then we apply Theorem \ref{thm:1.4} to see that there exist a positive real number $\rho$ and a finite subset $\mathcal{E} \subset \mathcal{A}$ such that 
\begin{gather*}
\liminf_{T \to\infty} \frac{1}{T} \meas
\left\{ \tau \in [0,T] ~\middle|~ \sup_{s \in K} |\zeta(s+i \tau, \alpha')-f(s)|<\frac{1}{k} \right\}
> 0
\end{gather*}
for any $\alpha' \in \mathcal{A}_\rho(\alpha) \setminus \mathcal{E}$, where $\rho$ and $\mathcal{E}$ depend on $\alpha,f,K$, and $k$. 
For any $k \geq1$, we can take  an element $\alpha_k$ in $\mathcal{A}_\rho(\alpha) \setminus \mathcal{E}$ with $|\alpha-\alpha_k|<1/k$ since $\mathcal{E}$ is a finite set. 
Therefore, Theorem \ref{thm:1} follows if we take $k_0(\epsilon)$ so that $k_0(\epsilon) \geq 1/\epsilon$ is satisfied. 
Let $1/2<\sigma_1<\sigma_2<1$. 
To deduce Theorem \ref{thm:2}, we note that the number of the zeros of $\zeta(s,\alpha)$ in $D$ is greater than $N_\alpha(\sigma_1,\sigma_2,T)$ for any $T>0$. 
Since $N_\alpha(\sigma_1,\sigma_2,T) \to\infty$ as $T \to\infty$ for infinitely many $\alpha \in \mathcal{A}$ by Theorem \ref{thm:1.6}, we obtain the conclusion.

\subsection{Key ideas for the proof}\label{sec:1.3}
Bagchi's method for the proof of the universality of $\zeta(s)$ in \cite{Bagchi1981} begins with constructing a certain random element $\zeta(s,\mathbb{X})$ related to the distribution of $\zeta(s)$. 
Then the probability measure 
\begin{gather}\label{eq:04252342}
P_T(A)
= \frac{1}{T} \meas \left\{ \tau \in [0,T] ~\middle|~ \zeta(s+i \tau) \in A \right\}
\end{gather}
converges weakly as $T \to\infty$ to the law of $\zeta(s,\mathbb{X})$. 
Furthermore, by showing that any function $f$ with nice properties is approximated by $\zeta(s,\mathbb{X})$ with positive probability, we obtain the universality theorem of $\zeta(s)$. 

The proof of Theorem \ref{thm:1.4} is derived by a modification of Bagchi's method. 
First, we construct a random Dirichlet series $\zeta(s,\mathbb{X}_\alpha)$ related to the distribution of $\zeta(s, \alpha)$, according to the method of the previous paper \cite{Mine2022+}. 
Then we prove that a probability measure defined analogously to \eqref{eq:04252342} for $\zeta(s,\alpha)$ converges weakly as $T \to\infty$ to the law of $\zeta(s,\mathbb{X}_\alpha)$. 
See Theorem \ref{thm:3.1}. 
Then Gonek's conjecture would follow if we could show that $f$ is approximated by $\zeta(s,\mathbb{X}_\alpha)$ with positive probability. 

Let $\zeta_N(s,\mathbb{X}_\alpha)$ denote a random Dirichlet polynomial obtained as a partial sum of $\zeta(s,\mathbb{X}_\alpha)$. 
It is quite hard to study $\zeta(s,\mathbb{X}_\alpha)$ and $\zeta_N(s,\mathbb{X}_\alpha)$ for $\alpha \in \mathcal{A}$ due to the lack of the linear independence of $\log(n_1+\alpha), \ldots, \log(n_k+\alpha)$ as described before. 
For this reason, we introduce another random Dirichlet polynomial $\zeta_N(s,\mathbb{Y}_\alpha)$ to show that any function $f$ with nice properties is approximated by $\zeta_N(s,\mathbb{Y}_\alpha)$ with positive probability. 
See Theorem \ref{thm:4.1} and Corollary \ref{cor:4.9}. 
Then, we observe that the error in a replacement of $\zeta_N(s,\mathbb{X}_\alpha)$ with $\zeta_N(s,\mathbb{Y}_\alpha)$ can be ignored in some sense except for a finite number of $\alpha \in \mathcal{A}$. 
This is a key idea in this paper for studying $\zeta(s,\mathbb{X}_\alpha)$ and $\zeta_N(s,\mathbb{X}_\alpha)$. 
The finite subset $\mathcal{E}$ in the statement of Theorem \ref{thm:1.4} comes from the exceptional $\alpha \in \mathcal{A}$ in the replacement of $\zeta_N(s,\mathbb{X}_\alpha)$ with $\zeta_N(s,\mathbb{Y}_\alpha)$.

\subsection{Organization of the paper}\label{sec:1.4}
The main content of this paper is to establish Theorem \ref{thm:1.4}. 
The proof is divided into the following three parts. 

\begin{itemize}
\item[1.] 
\textit{Limit theorem for $\zeta(s,\alpha)$ in the function space.} 
(Section \ref{sec:2} \& Section \ref{sec:3})

Let $H(D)$ be the space of analytic functions on $D$. 
The first step to the proof of universality is to study the distribution of $\zeta(s+i \tau,\alpha)$ for $\tau \in \mathbb{R}$ in the space $H(D)$. 
For this, we introduce in Section \ref{sec:2} a sequence $\{\mathbb{X}_\alpha(n)\}$ of random variables valued on the unit circle of $\mathbb{C}$ such that $\mathbb{X}_\alpha(n_1),\ldots,\mathbb{X}_\alpha(n_k)$ are independent if and only if the real numbers $\log(n_1+\alpha), \ldots, \log(n_k+\alpha)$ are linearly independent over $\mathbb{Q}$. 
Then we define a random Dirichlet series $\zeta(s,\mathbb{X}_\alpha)$ valued on $H(D)$ by using $\{\mathbb{X}_\alpha(n)\}$. 
In Section \ref{sec:3}, we prove the limit theorem (Theorem \ref{thm:3.1}) which asserts that  a probability measure defined analogously to \eqref{eq:04252342} for $\zeta(s,\alpha)$ converges weakly to the law of $\zeta(s,\mathbb{X}_\alpha)$. 

\item[2.] 
\textit{Covering theorem for the Bergman space.}
(Section \ref{sec:4})

Let $A^2(U)$ be the Bergman space for a bounded domain $U$ with $\overline{U} \subset D$. 
We introduce another sequence $\{\mathbb{Y}_\alpha(n)\}$ of random variables valued on the unit circle of $\mathbb{C}$, which satisfies that $\mathbb{Y}_\alpha(n_1),\ldots,\mathbb{Y}_\alpha(n_k)$ are independent for any distinct integers $n_1,\ldots,n_k \geq0$. 
Then we define a random Dirichlet polynomial $\zeta_N(s,\mathbb{Y}_\alpha)$ valued on $A^2(U)$ by using $\{\mathbb{Y}_\alpha(n)\}$. 
In Section \ref{sec:4}, we prove a covering theorem (Theorem \ref{thm:4.1}) which asserts that $A^2(U)$ is covered by neighborhoods of certain sets of Dirichlet polynomials related to $\zeta_N(s,\mathbb{Y}_\alpha)$. 
This theorem yields that any function $f \in A^2(U)$ can be approximated by $\zeta_N(s,\mathbb{Y}_\alpha)$ with positive probability (Corollary \ref{cor:4.9}), if $\rho$ is sufficiently small and $N$ is sufficiently large. 

\item[3.] 
\textit{Estimate of a conditional square mean value.}
(Section \ref{sec:5})

Corollary \ref{cor:4.9} implies that any $f \in A^2(U)$ is approximated by the Dirichlet polynomial $\zeta_N(s,\mathbb{Y}_\alpha)(\omega_0)$ with some sample $\omega_0$. 
Then we define an event $\Omega_0$ so that each of $\mathbb{X}_\alpha(n)$ is close to $\mathbb{Y}_\alpha(n)(\omega_0)$ for $0 \leq n \leq N$. 
By definition, $f$ can be approximated by $\zeta_N(s,\mathbb{X}_\alpha)(\omega)$ if $\omega \in \Omega_0$. 
In Section \ref{sec:5}, we prove an upper bound for the expected value of $\|\zeta(s,\mathbb{X}_\alpha)-\zeta_N(s,\mathbb{X}_\alpha)\|^2$ under the condition $\omega \in \Omega_0$ for $\alpha \in \mathcal{A} \setminus \mathcal{E}$, where $\mathcal{E}$ is a finite subset. 
The key idea described in Section \ref{sec:1.3} is closely related to this result. 
See Proposition \ref{prop:5.2} and Remark \ref{rem:5.3} for the details. 
Then we deduce that any function $f \in H(D)$ is approximated by $\zeta(s,\mathbb{X}_\alpha)$ with positive probability. 
Finally, we obtain the desired approximation of $f$ as in the statement of Theorem \ref{thm:1.4} by applying Theorem \ref{thm:3.1} and the Mergelyan approximation theorem. 
\end{itemize}

Theorems \ref{thm:1.5} and \ref{thm:1.6} are deduced from Theorem \ref{thm:1.4} by standard methods in the theory of universality. 
The proofs are completed in Section \ref{sec:6}.

\subsection*{Acknowledgments}
The author is especially grateful to Ram\={u}nas Garunk\v{s}tis for making me aware of his paper \cite{Garunkstis2005}. 
I would also like to thank Masatoshi Suzuki for valuable comments to improve the quality of the paper. 
The work of this paper was supported by JSPS Grant-in-Aid for Early-Career Scientists (Grant Number 24K16906).

\section{Random Dirichlet series related to $\zeta(s,\alpha)$}\label{sec:2}
Let $H(D)$ denote the space of analytic functions on $D$ equipped with the topology of uniform convergence on compact subsets. 
It is known that $H(D)$ is a complete metric space by the distance $d$ defined as follows. 
Let 
\begin{gather}\label{eq:04240245}
K_\nu
= \left\{s \in \mathbb{C} ~\middle|~ \frac{1}{2}+\frac{1}{5\nu} \leq \sigma \leq 1-\frac{1}{5\nu}, ~ -\nu \leq t \leq \nu \right\}
\end{gather}
for $\nu \geq1$. 
Then $\{K_\nu\}$ is a sequence of compact subsets of $D$ such that $K_\nu \subset K_{\nu+1}$ and $D= \bigcup_{\nu=1}^{\infty} K_\nu$. 
Furthermore, if $K$ is a compact subset of $D$, then $K \subset K_\nu$ holds for some $\nu \geq1$. 
For $f,g \in H(D)$, we define 
\begin{gather}\label{eq:04131758}
d(f,g)
= \sum_{\nu=1}^{\infty} 2^{-\nu} \frac{d_\nu(f,g)}{1+d_\nu(f,g)}, 
\qquad
d_\nu(f,g)
= \sup_{s \in K_\nu} |f(s)-g(s)|. 
\end{gather}
Denote by $\mathcal{B}(H(D))$ the Borel algebra of $H(D)$ with the topology described above. 
The purpose of this section is to introduce two random elements
\begin{gather}\label{eq:04120302}
\zeta(s,\mathbb{X}_\alpha)
= \sum_{n=0}^{\infty} \frac{\mathbb{X}_\alpha(n)}{(n+\alpha)^s}
\quad\text{and}\quad
\zeta(s,\mathbb{Y}_\alpha)
= \sum_{n=0}^{\infty} \frac{\mathbb{Y}_\alpha(n)}{(n+\alpha)^s}
\end{gather}
valued on the space $H(D)$. 
Here, $\{\mathbb{X}_\alpha(n)\}$ and $\{\mathbb{Y}_\alpha(n)\}$ are sequences of certain random variables valued on the unit circle $S^1= \{z \in \mathbb{C} \mid |z|=1\}$.

\subsection{Random variables $\mathbb{X}_\alpha(n)$ and $\mathbb{Y}_\alpha(n)$}\label{sec:2.1}
Let $\alpha$ be an algebraic number with $0<\alpha \leq1$. 
The construction of the random variable $\mathbb{X}_\alpha(n)$ is essentially the same as in the previous paper \cite{Mine2022+}. 
Let $\mathcal{K}=\mathbb{Q}(\alpha)$ be the algebraic field generated by $\alpha$, and denote by $\mathcal{O}_\mathcal{K}$ the ring of integers of $\mathcal{K}$. 
Let $h$ be the class number of $\mathcal{K}$. 
Recall that $\mathfrak{a}^h$ is a principal ideal of $\mathcal{K}$ for any ideal $\mathfrak{a}$. 
Then, for any prime ideal $\mathfrak{p}$, we fix an element $\varpi_{\mathfrak{p}} \in \mathcal{O}_\mathcal{K}$ such that $\mathfrak{p}^h=(\varpi_{\mathfrak{p}})$ and $\varpi_{\mathfrak{p}}>0$. 
Let $y \in K$ with $y>0$. 
The fractional principal ideal $(y)$ has the decomposition
\begin{gather*}
(y)
= \mathfrak{p}_1^{a_1} \cdots \mathfrak{p}_k^{a_k}, 
\end{gather*}
where $\mathfrak{p}_j$ are prime ideals, and $a_j \in \mathbb{Z}$ are uniquely determined by $y$. 
It yields
\begin{gather*}
(y)^h 
= \mathfrak{p}_1^{h a_1} \cdots \mathfrak{p}_k^{h a_k}
= (\varpi_{\mathfrak{p}_1}^{a_1} \cdots \varpi_{\mathfrak{p}_k}^{a_k}). 
\end{gather*}
Next, we fix a fundamental system $(u_1, \ldots, u_d)$ of units of $\mathcal{O}_\mathcal{K}$ such that $u_j>0$ for every $j$. 
Then $y^h$ has the decomposition
\begin{gather}\label{eq:04120008}
y^h
= \varpi_{\mathfrak{p}_1}^{a_1} \cdots \varpi_{\mathfrak{p}_k}^{a_k} u_1^{b_1} \cdots u_d^{b_d}, 
\end{gather}
where $b_j \in \mathbb{Z}$ are uniquely determined by $y$. 
Let 
\begin{gather*}
\Lambda_1
= \{\varpi_\mathfrak{p} \mid \text{$\mathfrak{p}$ is a prime ideal of $\mathcal{K}$} \}
\cup \{u_1, u_2, \ldots, u_d\}. 
\end{gather*}
For $\lambda \in \Lambda_1$, we define $\ord(y,\lambda)$ as 
\begin{gather}\label{eq:04120112}
\ord(y,\lambda)
=
\begin{cases}
a_j
& \text{if $\lambda=\varpi_{\mathfrak{p}_j}$ for some $\varpi_{\mathfrak{p}_j}$ in \eqref{eq:04120008}}, 
\\
b_j
& \text{if $\lambda=u_j$ for some $u_j$ in \eqref{eq:04120008}}, 
\\
0
& \text{otherwise}. 
\end{cases}
\end{gather}
By the uniqueness of $a_j$ and $b_j$, we have the identity
\begin{gather}\label{eq:04120224}
\ord(y_1y_2,\lambda)
= \ord(y_1,\lambda)+\ord(y_2,\lambda)
\end{gather}
for any $\lambda \in \Lambda_1$ and any $y_1,y_2 \in \mathcal{K}$ with $y_1,y_2>0$. 
We are now ready to define the random variable $\mathbb{X}_\alpha(n)$. 
For any $s,t \in \mathbb{R}$ with $0<t-s \leq 2\pi$, we define
\begin{gather*}
\mathbf{m}(A(s,t))
= \frac{t-s}{2\pi}, 
\end{gather*}
where $A(s,t)$ is the arc of the unit circle $S^1$ defined as
\begin{gather}\label{eq:04212057}
A(s,t)
= \{e^{i \theta} \mid s<\theta<t \}.
\end{gather}
Then $\mathbf{m}$ is extended to a probability measure on $(S^1, \mathcal{B}(S^1))$, where $\mathcal{B}(S^1)$ denotes the Borel algebra of $S^1$ with the usual topology. 
Let $\Omega_1=\prod_{\lambda \in \Lambda_1} S^1$, and denote by $\mathcal{F}_1$ the product $\sigma$-algebra $\bigotimes_{\lambda_ \in \Lambda_1} \mathcal{B}(S^1)$. 
By the Kolmogorov extension theorem, there exists a probability measure $\mathbf{P}_1$ on $(\Omega_1, \mathcal{F}_1)$ such that 
\begin{gather}\label{eq:04121502}
\mathbf{P}_1 \left(\left\{\omega=\{\omega_\lambda\} \in \Omega_1 ~\middle|~ 
\text{$\omega_{\lambda_j} \in A(s_j,t_j)$ for $j=1,\ldots,k$} \right\}\right)
= \prod_{j=1}^{k} \mathbf{m}(A(s_j,t_j))
\end{gather}
for any finite subset $\{\lambda_1,\ldots,\lambda_k\} \subset \Lambda_1$ and any $s_j,t_j \in \mathbb{R}$ with $0<t_j-s_j \leq 2\pi$. 
Denote by $\mathcal{X}(\lambda): \Omega_1 \to S^1$ the $\lambda$-th projection $\omega \mapsto \omega_\lambda$. 
Then we define
\begin{gather*}
\mathbb{X}_\alpha(n)
= \prod_{\lambda \in \Lambda_1} \mathcal{X}(\lambda)^{\ord(n+\alpha, \lambda)}
\end{gather*}
for $n \geq0$ by using $\ord(y,\lambda)$ defined as \eqref{eq:04120112}. 
By definition, it is an $S^1$-valued random variable for any $n \geq0$. 

On the other hand, the random variable $\mathbb{Y}_\alpha(n)$ is defined as follows. 
Let 
\begin{gather*}
\Lambda_2
= \{n \in \mathbb{Z} \mid n \geq0\}. 
\end{gather*}
Let $\Omega_2=\prod_{n \in \Lambda_2} S^1$, and denote by $\mathcal{F}_2$ the product $\sigma$-algebra $\bigotimes_{n \in \Lambda_2} \mathcal{B}(S^1)$. 
By the Kolmogorov extension theorem, there exists a probability measure $\mathbf{P}_2$ on $(\Omega_2, \mathcal{F}_2)$ such that 
\begin{gather*}
\mathbf{P}_2 \left(\left\{\omega=\{\omega_n\} \in \Omega_2 ~\middle|~ 
\text{$\omega_{n_j} \in A(s_j,t_j)$ for $j=1,\ldots,k$} \right\}\right)
= \prod_{j=1}^{k} \mathbf{m}(A(s_j,t_j))
\end{gather*}
for any finite subset $\{n_1,\ldots,n_k\} \subset \Lambda_2$ and any $s_j,t_j \in \mathbb{R}$ with $0<t_j-s_j \leq 2\pi$. 
Then we define the $S^1$-valued random variable $\mathbb{Y}_\alpha(n)$ as the $n$-th projection $\omega \mapsto \omega_n$ for $n \geq0$.  

For convenience, we regard $\mathbb{X}_\alpha(n)$ and $\mathbb{Y}_\alpha(n)$ as random variables defined on the same probability space $(\Omega,\mathcal{F},\mathbf{P})$ as follows. 
Let $\Omega=\Omega_1 \times \Omega_2$, $\mathcal{F}=\mathcal{F}_1 \otimes \mathcal{F}_2$, and $\mathbf{P}=\mathbf{P}_1 \otimes \mathbf{P}_2$. 
Then we put
\begin{gather*}
\mathbb{X}_\alpha(n)(\omega^1,\omega^2)= \mathbb{X}_\alpha(n)(\omega^1)
\quad\text{and}\quad
\mathbb{Y}_\alpha(n)(\omega^1,\omega^2)= \mathbb{Y}_\alpha(n)(\omega^2)
\end{gather*}
for $(\omega^1,\omega^2) \in \Omega$ with $\omega^1 \in \Omega_1$ and $\omega^2 \in \Omega_2$. 
In general, the expected value of a $\mathbb{C}$-valued random variable $\mathcal{X}$ defined on $(\Omega,\mathcal{F},\mathbf{P})$ is denoted by 
\begin{gather*}
\mathbf{E}[\mathcal{X}]
= \int_{\Omega} \mathcal{X} \,d\mathbf{P}
\end{gather*}
as usual. 
We abbreviate $\mathbf{P}\left(\{\omega \in \Omega \mid \mathcal{X}(\omega) \in A \}\right)$ as $\mathbf{P}\left(\mathcal{X} \in A\right)$ for simplicity, which is called the law of a random element $\mathcal{X}$. 

\begin{lemma}\label{lem:2.1}
Let $\alpha$ be an algebraic number with $0<\alpha \leq1$. 
Then we have
\begin{gather*}
\mathbf{E}\left[\mathbb{X}_\alpha(n_1)^{m_1} \cdots \mathbb{X}_\alpha(n_k)^{m_k}\right]
=
\begin{cases}
1
& \text{if $(n_1+\alpha)^{m_1}\cdots(n_k+\alpha)^{m_k}=1$},
\\
0
& \text{otherwise}
\end{cases}
\end{gather*}
for any integers $n_1, \ldots, n_k \geq0$ and $m_1, \ldots, m_k \in \mathbb{Z}$. 
Furthermore, 
\begin{gather*}
\mathbf{E}\left[\mathbb{Y}_\alpha(n_1)^{m_1} \cdots \mathbb{Y}_\alpha(n_k)^{m_k}\right]
=
\begin{cases}
1
& \text{if $m_1=\cdots=m_k=0$},
\\
0
& \text{otherwise}
\end{cases}
\end{gather*}
for any distinct integers $n_1, \ldots, n_k \geq0$ and $m_1, \ldots, m_k \in \mathbb{Z}$. 
\end{lemma}

\begin{proof}
The first result has already been proved in \cite{Mine2022+}, but we present the full proof here because of its importance. 
Put
\begin{gather*}
y
= (n_1+\alpha)^{m_1}\cdots(n_k+\alpha)^{m_k}. 
\end{gather*}
Applying \eqref{eq:04120224}, we obtain 
\begin{gather*}
\ord(y,\lambda)
= m_1 \ord(n_1+\alpha,\lambda) +\cdots+ m_k \ord(n_k+\alpha,\lambda)
\end{gather*}
for any $\lambda \in \Lambda_1$. 
Therefore the formula
\begin{align*}
\mathbb{X}_\alpha(n_1)^{m_1} \cdots \mathbb{X}_\alpha(n_k)^{m_k}
&= \prod_{\lambda \in \Lambda_1} 
\mathcal{X}(\lambda)^{m_1 \ord(n_1+\alpha, \lambda) +\cdots+ m_k \ord(n_k+\alpha, \lambda)} \\
&= \prod_{\lambda \in \Lambda_1} 
\mathcal{X}(\lambda)^{\ord(y,\lambda)}
\end{align*}
holds by the definition of $\mathbb{X}_\alpha(n_j)$. 
Then the expected value is calculated as
\begin{align}\label{eq:04120238}
\mathbf{E}\left[\mathbb{X}_\alpha(n_1)^{m_1} \cdots \mathbb{X}_\alpha(n_k)^{m_k}\right]
&= \int_{\Omega_1} \prod_{\lambda \in \Lambda_1} \mathcal{X}(\lambda)^{\ord(y,\lambda)} \,d\mathbf{P}_1 \\
&= \prod_{\lambda \in \Lambda_1} \int_{S_1} \omega_\lambda^{\ord(y,\lambda)} \,d\mathbf{m}(\omega_\lambda) \nonumber
\end{align}
by \eqref{eq:04121502} and Fubini's theorem. 
We see that 
\begin{gather}\label{eq:04120245}
\int_{S_1} \omega_\lambda^q \,d\mathbf{m}(\omega_\lambda)
= \frac{1}{2\pi} \int_{0}^{2\pi} e^{i q \theta} \,d \theta
= 
\begin{cases}
1
& \text{if $q=0$},
\\
0
& \text{otherwise} 
\end{cases}
\end{gather}
for any $q \in \mathbb{Z}$. 
Hence, the right-hand side of \eqref{eq:04120238} is
\begin{gather*}
\prod_{\lambda \in \Lambda_1} \int_{S_1} \omega_\lambda^{\ord(y,\lambda)} \,d\mathbf{m}(\omega_\lambda)
=
\begin{cases}
1
& \text{if $\ord(y,\lambda)=0$ for any $\lambda \in \Lambda_1$}, 
\\
0
& \text{otherwise}. 
\end{cases}
\end{gather*}
The condition that $\ord(y,\lambda)=0$ for any $\lambda \in \Lambda_1$ is equivalent to $y^h=1$ by \eqref{eq:04120008}. 
Note that $y^h=1$ if and only if $y=1$ since $y$ is a positive real number. 
Therefore, the first result follows. 
The second result is proved more easily. 
We have 
\begin{align*}
\mathbf{E}\left[\mathbb{Y}_\alpha(n_1)^{m_1} \cdots \mathbb{Y}_\alpha(n_k)^{m_k}\right]
&= \int_{\Omega_2} \prod_{j=1}^{k} \mathbb{Y}_\alpha(n_j)^{m_j} \,d\mathbf{P}_2 \\
&= \prod_{j=1}^{k} \int_{S_1} \omega_{n_j}^{m_j} \,d\mathbf{m}(\omega_{n_j}) 
\end{align*}
by the definition of $\mathbb{Y}_\alpha(n_j)$. 
Using \eqref{eq:04120245}, we obtain the conclusion. 
\end{proof}

Moreover, it was proved in \cite{Mine2022+} that $\mathbb{X}_\alpha(n)$ is uniformly distributed on $S^1$ if $0<\alpha<1$, and that $\mathbb{X}_\alpha(n_1), \ldots, \mathbb{X}_\alpha(n_k)$ are independent if and only if the real numbers $\log(n_1+\alpha), \ldots, \log(n_k+\alpha)$ are linearly independent over $\mathbb{Q}$. 
It is rather clear that $\mathbb{Y}_\alpha(n)$ is uniformly distributed on $S^1$, and that $\mathbb{Y}_\alpha(n_1),\ldots,\mathbb{Y}_\alpha(n_k)$ are independent for any distinct integers $n_1,\ldots,n_k \geq0$.

\subsection{Random Dirichlet series $\zeta(s,\mathbb{X}_\alpha)$ and $\zeta(s,\mathbb{Y}_\alpha)$}\label{sec:2.2}
Using the random variables $\mathbb{X}_\alpha(n)$ and $\mathbb{Y}_\alpha(n)$ constructed in Section \ref{sec:2.1}, we define the Dirichlet polynomials
\begin{gather*}
\zeta_N(s,\mathbb{X}_\alpha)
= \sum_{n=0}^{N} \frac{\mathbb{X}_\alpha(n)}{(n+\alpha)^s}
\quad\text{and}\quad
\zeta_N(s,\mathbb{Y}_\alpha)
= \sum_{n=0}^{N} \frac{\mathbb{Y}_\alpha(n)}{(n+\alpha)^s}
\end{gather*}
for $N \geq0$. 
To begin with, we check that they are random elements valued on the space $H(D)$. 

\begin{lemma}\label{lem:2.2}
Let $\alpha$ be an algebraic number with $0<\alpha \leq1$. 
Then $\zeta_N(s,\mathbb{X}_\alpha)$ and $\zeta_N(s,\mathbb{Y}_\alpha)$ are random elements valued on $H(D)$ for any $N \geq0$. 
\end{lemma}

\begin{proof}
Note that the map $\psi: \prod_{n=0}^{N} S^1 \to H(D)$ defined as
\begin{gather*}
(\gamma_0, \ldots, \gamma_N) \mapsto \sum_{n=0}^{N} \frac{\gamma_n}{(n+\alpha)^s}
\end{gather*}
is continuous. 
Since $\zeta_N(s,\mathbb{X}_\alpha)$ and $\zeta_N(s,\mathbb{Y}_\alpha)$ are obtained as compositions of $\psi$ and the $S^1$-valued random variables $\mathbb{X}_\alpha(n)$ and $\mathbb{Y}_\alpha(n)$, they are random elements valued on $H(D)$. 
\end{proof}

To consider the convergences of the random Dirichlet series $\zeta(s,\mathbb{X}_\alpha)$ and $\zeta(s,\mathbb{Y}_\alpha)$ as in \eqref{eq:04120302}, we apply the following version of the Menshov--Rademacher theorem. 

\begin{lemma}[Menshov--Rademacher theorem]\label{lem:2.3}
Let $\{\mathcal{X}_n\}$ be a sequence of $\mathbb{C}$-valued random variables defined on the probability space $(\Omega, \mathcal{F}, \mathbf{P})$. 
Then the random series 
\begin{gather*}
\sum_{n=1}^{\infty} a_n \mathcal{X}_n
\end{gather*}
converges almost surely if the following conditions are satisfied.  
\begin{itemize}
\item[$(\mathrm{i})$] 
$\mathbf{E}[\mathcal{X}_n]=0$ for any $n \geq1$. 
\item[$(\mathrm{ii})$] 
The sequence $\{\mathcal{X}_n\}$ is orthonormal in the sense that
\begin{gather*}
\mathbf{E}[\mathcal{X}_m \overline{\mathcal{X}_n}]
=
\begin{cases}
1
& \text{if $m=n$},
\\
0
& \text{otherwise}. 
\end{cases}
\end{gather*}
\item[$(\mathrm{iii})$] 
We have 
\begin{gather*}
\sum_{n=1}^{\infty} |a_n|^2 (\log{n})^2
< \infty. 
\end{gather*}
\end{itemize}
\end{lemma}

\begin{proof}
See \cite[Theorem B.10.5]{Kowalski2021} for a proof. 
\end{proof}

\begin{proposition}\label{prop:2.4}
Let $\alpha$ be an algebraic number with $0<\alpha \leq1$. 
Then $\zeta(s,\mathbb{X}_\alpha)$ and $\zeta(s,\mathbb{Y}_\alpha)$ are random elements valued on $H(D)$. 
\end{proposition}

\begin{proof}
By Lemma \ref{lem:2.1}, we can check that the sequence $\{\mathbb{X}_\alpha(n)\}$ satisfies conditions $(\mathrm{i})$ and $(\mathrm{ii})$ of Lemma \ref{lem:2.3}. 
Furthermore, we have
\begin{gather*}
\sum_{n=0}^{\infty} \left|\frac{1}{(n+\alpha)^s}\right|^2 (\log{n})^2
= \sum_{n=0}^{\infty} \frac{(\log{n})^2}{(n+\alpha)^{2\sigma}}
< \infty
\end{gather*}
on the half-plane $\sigma>1/2$. 
Hence Lemma \ref{lem:2.3} yields that the random Dirichlet series $\zeta(s,\mathbb{X}_\alpha)$ converges for $\sigma>1/2$ almost surely. 
In other words, if we define
\begin{gather*}
\Omega_\mathrm{c}
= \{\omega \in \Omega \mid \text{$\zeta_N(s,\mathbb{X}_\alpha)(\omega)$ converges as $N \to\infty$ for $\sigma>1/2$} \}, 
\end{gather*}
then $\mathbf{P}(\Omega_\mathrm{c})=1$. 
For $\delta>0$, we also define 
\begin{gather*}
\Omega_{\mathrm{u},\delta}
= \{\omega \in \Omega \mid \text{$\zeta_N(s,\mathbb{X}_\alpha)(\omega)$ converges uniformly as $N \to\infty$ for $\sigma \geq1/2+\delta$} \}. 
\end{gather*}
Then $\Omega_\mathrm{c} \subset \Omega_{\mathrm{u},\delta}$ is valid for any $\delta>0$. 
Indeed, if a Dirichlet series converges for $\sigma>1/2$, then it converges uniformly for $\sigma \geq 1/2+\delta$. 
See \cite[Corollary 2.1.3]{Laurincikas1996a} for a proof. 
Let $\Omega_\mathrm{u}=\bigcap_{\delta>0} \Omega_{\mathrm{u},\delta}$. 
Then 
\begin{gather*}
\zeta(s,\mathbb{X}_\alpha)(\omega)
= \lim_{N \to\infty} \zeta_N(s,\mathbb{X}_\alpha)(\omega)
\end{gather*}
belongs to the space $H(D)$ for any $\omega \in \Omega_\mathrm{u}$. 
Furthermore, we obtain that $\mathbf{P}(\Omega_\mathrm{u})=1$ since $\mathbf{P}(\Omega_\mathrm{c})=1$ and $\Omega_\mathrm{c} \subset \Omega_{\mathrm{u},\delta}$ for any $\delta>0$. 
Thus $\zeta(s,\mathbb{X}_\alpha)$ defines a map from $\Omega$ to $H(D)$, which is a random element since $\zeta_N(s,\mathbb{X}_\alpha)$ is a random element for any $N \geq0$. 
The proof for $\zeta(s,\mathbb{Y}_\alpha)$ is completely the same as that for $\zeta(s,\mathbb{X}_\alpha)$. 
\end{proof}

\section{Limit theorem for $\zeta(s,\alpha)$ in the space $H(D)$}\label{sec:3}
Let $\alpha$ be an algebraic number with $0<\alpha \leq1$. 
We define 
\begin{gather*}
P_{\alpha,T}(A)
= \frac{1}{T} \meas \left\{\tau \in [0,T] ~\middle|~ \zeta(s+i \tau,\alpha) \in A \right\}, \\
Q_\alpha(A)
= \mathbf{P}(\zeta(s,\mathbb{X}_\alpha) \in A) 
\end{gather*}
for $A \in \mathcal{B}(H(D))$, where $\zeta(s,\mathbb{X}_\alpha)$ is the random Dirichlet series defined in Section \ref{sec:2}. 
Then $P_{\alpha,T}$ and $Q_\alpha$ are probability measures on $(H(D), \mathcal{B}(H(D)))$. 
In this section, we prove the following limit theorem. 

\begin{theorem}\label{thm:3.1}
Let $\alpha$ be an algebraic number with $0<\alpha \leq1$. 
Then the probability measure $P_{\alpha,T}$ converges weakly to $Q_\alpha$ as $T \to\infty$. 
\end{theorem}

This is a generalization of Bagchi's limit theorem of \cite{Bagchi1981}, which was proved for the Riemann zeta-function $\zeta(s)$. 
The proof of Theorem \ref{thm:3.1} is largely based on Bagchi's method, while we adopt new ideas from Kowalski \cite{Kowalski2017, Kowalski2021} in part.

\subsection{Fourier analysis on a torus}\label{sec:3.1}
Let $\mathcal{T}^k$ be the $k$-dimensional torus given by
\begin{gather*}
\mathcal{T}^k
= \prod_{j=1}^{k} S^1. 
\end{gather*}
Denote by $\mathcal{B}(\mathcal{T}^k)$ the Borel algebra of $\mathcal{T}^k$ with the product topology. 
Let $\mu$ be any probability measure on $(\mathcal{T}^k, \mathcal{B}(\mathcal{T}^k))$, and define its Fourier transform as 
\begin{gather*}
g(\underline{m})
= \int_{\mathcal{T}^k} \prod_{j=1}^{k} \gamma_j^{m_j} \,d \mu (\underline{\gamma}) 
\end{gather*}
for $\underline{m}=(m_1,\ldots,m_k) \in \mathbb{Z}^k$. 
Then the following result holds for the torus $\mathcal{T}^k$. 

\begin{lemma}\label{lem:3.2}
Let $\{\mu_n\}$ be a sequence of probability measures on $(\mathcal{T}^k, \mathcal{B}(\mathcal{T}^k))$ whose Fourier transforms are denoted by $g_n(\underline{m})$. 
If the limit 
\begin{gather*}
g(\underline{m})
= \lim_{n \to\infty} g_n(\underline{m})
\end{gather*}
exists for any $\underline{m} \in \mathbb{Z}^k$, then there exists a probability measure $\nu$ on $(\mathcal{T}^k, \mathcal{B}(\mathcal{T}^k))$ such that the following properties are satisfied. 
\begin{itemize}
\item[$(\mathrm{i})$]
The probability measure $\mu_n$ converges weakly to $\nu$ as $n \to\infty$. 
\item[$(\mathrm{ii})$]
The Fourier transform of $\nu$ is equal to $g(\underline{m})$. 
\end{itemize}
\end{lemma}

\begin{proof}
This is a special case of a more general result \cite[Theorem 1.4.2]{Heyer1977} proved for every locally compact Abelian group. 
See also \cite[Theorem 1.3.19]{Laurincikas1996a}. 
\end{proof}

Let $\alpha$ be an algebraic number with $0<\alpha \leq1$. 
Define the probability measures $\mu_{\alpha,T}$ and $\nu_\alpha$ on $(\mathcal{T}^k, \mathcal{B}(\mathcal{T}^k))$ by letting
\begin{gather*}
\mu_{\alpha,T}(A)
= \frac{1}{T} \meas \left\{\tau \in [0,T] ~\middle|~ \left((n_1+\alpha)^{-i \tau},\ldots,(n_k+\alpha)^{-i \tau}\right) \in A \right\}, \\
\nu_\alpha(A)
= \mathbf{P}\left( \left(\mathbb{X}_\alpha(n_1),\ldots,\mathbb{X}_\alpha(n_k)\right) \in A \right) 
\end{gather*}
for $A \in \mathcal{B}(\mathcal{T}^k)$, where $n_1,\ldots,n_k$ are non-negative integers. 
We begin by the proof of the following auxiliary result. 

\begin{proposition}\label{prop:3.3}
Let $\alpha$ be an algebraic number with $0<\alpha \leq1$. 
Then the probability measure $\mu_{\alpha,T}$ converges weakly to $\nu_\alpha$ as $T \to\infty$. 
\end{proposition}

\begin{proof}
Denote by $g_{\alpha,T}(\underline{m})$ the Fourier transform of $\mu_{\alpha,T}$. 
Then it is calculated as 
\begin{align*}
g_{\alpha,T}(\underline{m})
&= \int_{\mathcal{T}^k} \prod_{j=1}^{k} \gamma_j^{m_j} \,d \mu_{\alpha,T} (\underline{\gamma}) \\
&= \frac{1}{T} \int_{0}^{T} \left\{(n_1+\alpha)^{m_1}\cdots(n_k+\alpha)^{m_k}\right\}^{-i \tau} \,d \tau
\end{align*}
for any $\underline{m}=(m_1,\ldots,m_k) \in \mathbb{Z}^k$. 
If $(n_1+\alpha)^{m_1}\cdots(m_k+\alpha)^{m_k}=1$, then we have
\begin{gather*}
g_{\alpha,T}(\underline{m})
= \frac{1}{T} \int_{0}^{T} \,d \tau
= 1
\end{gather*}
for any $T>0$. 
Otherwise, we obtain that
\begin{gather*}
g_{\alpha,T}(\underline{m})
= \frac{1}{-i \Delta} \frac{\left\{(n_1+\alpha)^{m_1}\cdots(n_k+\alpha)^{m_k}\right\}^{-i T}-1}{T}, 
\end{gather*}
where $\Delta=m_1 \log(n_1+\alpha)+\cdots+m_k \log(n_k+\alpha) \neq 0$. 
As a result, we derive the limit formula
\begin{gather*}
\lim_{T \to\infty} g_{\alpha,T}(\underline{m})
= g_\alpha(\underline{m})
:=
\begin{cases}
1
& \text{if $(n_1+\alpha)^{m_1}\cdots(m_k+\alpha)^{m_k}=1$},
\\
0
& \text{otherwise}
\end{cases}
\end{gather*}
for any $\underline{m}=(m_1,\ldots,m_k) \in \mathbb{Z}^k$. 
Hence, by Lemma \ref{lem:3.2}, the probability measure $\mu_{\alpha,T}$ converges weakly to some probability measure on $(\mathcal{T}^k, \mathcal{B}(\mathcal{T}^k))$ whose Fourier transform is equal to $g_\alpha(\underline{m})$. 
Furthermore, the Fourier transform of $\nu_\alpha$ is  
\begin{gather*}
\int_{\mathcal{T}^k} \prod_{j=1}^{k} \gamma_j^{m_j} \,d \nu_\alpha (\underline{\gamma})
= \mathbf{E} \left[ \mathbb{X}_\alpha(n_1)^{m_1} \cdots \mathbb{X}_\alpha(n_k)^{m_k} \right]
= g_\alpha(\underline{m})
\end{gather*}
for any $\underline{m}=(m_1,\ldots,m_k) \in \mathbb{Z}^k$ by Lemma \ref{lem:2.1}. 
Therefore, the limit measure of $\mu_{\alpha, T}$ is equal to $\nu_\alpha$, and the proof is completed.  
\end{proof}

\subsection{Approximation by smoothed partial sums}\label{sec:3.2}
From now, we fix an infinitely differentiable function $\phi:[0,\infty) \to [0,1]$ with compact support such that $\phi(0)=1$. 
The Mellin transform of $\phi$ is defined by
\begin{gather*}
\widehat{\phi}(w)
= \int_{0}^{\infty} \phi(x) x^w \,\frac{dx}{x}
\end{gather*}
for $\RE(w)>0$. 
Then we collect standard properties of $\widehat{\phi}(w)$ as follows. 
For proofs, see \cite[Proposition A.3.1]{Kowalski2021}. 
Firstly, $\widehat{\phi}(w)$ extends to a meromorphic function on the half-plane $\RE(w)>-1$ only with a simple pole at $w=0$ with residue $1$. 
Secondly, $\widehat{\phi}(w)$ has rapid decay on the strip $a \leq \RE(w) \leq b$ with $a>-1$ in the following sense. 
Let $k \geq1$ and $b>a>-1$. 
Then we have 
\begin{gather}\label{eq:04150149}
|\widehat{\phi}(w)| 
\ll (|v|+1)^{-k}
\end{gather}
for any $w=u+iv \in \mathbb{C}$ with $a \leq u \leq b$ and $|v| \geq1$, where the implied constant depends only on $a,b,k$. 
Lastly, the inversion formula
\begin{gather}\label{eq:04151549}
\phi(x)
= \frac{1}{2\pi i} \int_{c-i \infty}^{c+i \infty} \widehat{\phi}(w) x^{-w} \,dw
\end{gather}
holds for any $c>0$. 
Let $\alpha$ be an algebraic number with $0<\alpha \leq1$. 
Then we define 
\begin{align*}
Z_N(s,\alpha)
&= \sum_{n=0}^{\infty} (n+\alpha)^{-s} \phi \left(\frac{n+\alpha}{N}\right), \\
Z_N(s,\mathbb{X}_\alpha)
&= \sum_{n=0}^{\infty} \frac{\mathbb{X}_\alpha(n)}{(n+\alpha)^s} \phi \left(\frac{n+\alpha}{N}\right) 
\end{align*}
for $N \geq1$. 
Since $\phi((n+\alpha)/N)=0$ for sufficiently large $n$, we see that $Z_N(s,\alpha)$ is a function belonging to the space $H(D)$, and that $Z_N(s,\mathbb{X}_\alpha)$ is a random element valued on $H(D)$. 
Then we prove two preliminary lemmas. 

\begin{lemma}\label{lem:3.4}
Let $\alpha$ be an algebraic number with $0<\alpha \leq1$. 
Let $s=\sigma+it$ be a complex number with $1/2<\sigma \leq1$, and suppose $0<\delta<\sigma-1/2$. 
Then we have 
\begin{gather*}
\zeta(s,\alpha)
= Z_N(s,\alpha)
- \frac{1}{2\pi i} \int_{-\delta-i \infty}^{-\delta+i \infty} \zeta(s+w,\alpha) \widehat{\phi}(w) N^w \,dw
- \widehat{\phi}(1-s) N^{1-s} 
\end{gather*}
for any $N \geq1$. 
Furthermore, we have
\begin{gather*}
\zeta(s,\mathbb{X}_\alpha)
= Z_N(s,\mathbb{X}_\alpha)
- \frac{1}{2\pi i} \int_{-\delta-i \infty}^{-\delta+i \infty} \zeta(s+w,\mathbb{X}_\alpha) \widehat{\phi}(w) N^w \,dw
\end{gather*}
almost surely for any $N \geq1$. 
\end{lemma}

\begin{proof}
Let $c$ be a positive real number with $\sigma+c>1$. 
By inversion formula \eqref{eq:04151549}, the function $Z_N(s,\alpha)$ is represented as
\begin{align*}
Z_N(s,\alpha)
&= \sum_{n=0}^{\infty} (n+\alpha)^{-s} \frac{1}{2\pi i} \int_{c-i \infty}^{c+i \infty} \widehat{\phi}(w) \left(\frac{n+\alpha}{N}\right)^{-w} \,dw \\
&= \frac{1}{2\pi i} \sum_{n=0}^{\infty} \int_{c-i \infty}^{c+i \infty} (n+\alpha)^{-(s+w)} \widehat{\phi}(w) N^w \,dw. 
\end{align*}
By Fubini's theorem, we obtain
\begin{gather*}
Z_N(s,\alpha)
= \frac{1}{2\pi i} \int_{c-i \infty}^{c+i \infty} \zeta(s+w,\alpha) \widehat{\phi}(w) N^w \,dw
\end{gather*}
since $\RE(s+w)>1$ for $\RE(w)=c$. 
Then we shift the contour to $\RE(w)=-\delta$ with $0<\delta<\sigma-1/2$. 
Note that we come across poles of $\zeta(s+w,\alpha) \widehat{\phi}(w) N^w$ only at $w=0, 1-s$, which are simple. 
Calculating the residues, we derive
\begin{gather*}
Z_N(s,\alpha)
= \frac{1}{2\pi i} \int_{-\delta-i \infty}^{-\delta+i \infty} \zeta(s+w,\alpha) \widehat{\phi}(w) N^w \,dw
+ \zeta(s,\alpha) 
+ \widehat{\phi}(1-s) N^{1-s}. 
\end{gather*}
This yields the desired formula of $\zeta(s,\alpha)$. 
As for $Z_N(s,\mathbb{X}_\alpha)$, we have 
\begin{gather*}
Z_N(s,\mathbb{X}_\alpha)(\omega)
= \frac{1}{2\pi i} \int_{c-i \infty}^{c+i \infty} \zeta(s+w,\mathbb{X}_\alpha)(\omega) \widehat{\phi}(w) N^w \,dw
\end{gather*}
for $\omega \in \Omega_\mathrm{u}$ in a similar way, where $\Omega_\mathrm{u}$ is the same as in the proof of Proposition \ref{prop:2.4}. 
Note that $\zeta(s+w,\mathbb{X}_\alpha)(\omega)$ remains holomorphic for any $\omega \in \Omega_\mathrm{u}$, while shifting the contour to $\RE(w)=-\delta$ with $0<\delta<\sigma-1/2$. 
Therefore, the formula
\begin{gather*}
Z_N(s,\mathbb{X}_\alpha)(\omega)
= \frac{1}{2\pi i} \int_{-\delta-i \infty}^{-\delta+i \infty} \zeta(s+w,\mathbb{X}_\alpha)(\omega) \widehat{\phi}(w) N^w \,dw
+ \zeta(s,\mathbb{X}_\alpha)(\omega) 
\end{gather*}
holds for $\omega \in \Omega_\mathrm{u}$. 
Since $\mathbf{P}(\Omega_\mathrm{u})=1$, we obtain the conclusion. 
\end{proof}

\begin{lemma}\label{lem:3.5}
Let $\alpha$ be an algebraic number with $0<\alpha \leq1$. 
Let $s=\sigma+it$ be a complex number with $1/2<\sigma<1$. 
Then we have
\begin{gather*}
\frac{1}{T} \int_{0}^{T} |\zeta(s+i \tau, \alpha)| \,d \tau
\ll \frac{1}{\sqrt{2\sigma-1}} \left(1+\frac{|t|}{T}\right)
+ \frac{1}{1-\sigma} \left(1+\frac{|t|}{T}\right)
\end{gather*}
for any $T \geq 3$. 
Furthermore, we have 
\begin{gather*}
\mathbf{E} \left[|\zeta(s,\mathbb{X}_\alpha)|\right]
\ll \frac{1}{\sqrt{2\sigma-1}}. 
\end{gather*}
Here, the implied constants depend only on $\alpha$. 
\end{lemma}

\begin{proof}
Let $1/2<\sigma<1$ be a real number. 
To begin with, we show that 
\begin{gather}\label{eq:04170059}
\frac{1}{V} \int_{-V}^{V} |\zeta(\sigma+iv,\alpha)| \,dv
\ll \frac{1}{\sqrt{2\sigma-1}} 
+ \frac{1}{1-\sigma}
\end{gather}
for any $V \geq 3$, where the implied constant depends only on $\alpha$. 
We use the following approximate formula of $\zeta(\sigma+iv,\alpha)$. 
For any $2\pi \leq |v| \leq \pi V$, we have 
\begin{gather*}
\zeta(\sigma+i v,\alpha)
= \sum_{0 \leq n \leq V} \frac{1}{(n+\alpha)^{\sigma+iv}}
+ \frac{V^{1-(\sigma+iv)}}{\sigma+iv-1}
+ O\left(V^{-\sigma}\right)
\end{gather*}
with an absolute implied constant. 
See \cite[Theorem III.2.1]{KaratsubaVoronin1992} for a proof. 
Then we obtain the estimate
\begin{align}\label{eq:04170044}
&\frac{1}{V} \int_{2\pi}^{V} |\zeta(\sigma+iv,\alpha)| \,dv \\
&\leq \frac{1}{V} \int_{2\pi}^{V} \left|\sum_{0 \leq n \leq V} \frac{1}{(n+\alpha)^{\sigma+iv}}\right| \,dv
+ \frac{1}{V}  \int_{2\pi}^{V} \left|\frac{V^{1-(\sigma+iv)}}{\sigma+iv-1}\right| \,dv
+ C V^{-\sigma} \nonumber\\
&\leq \left\{\frac{1}{V} \int_{0}^{V} \left|\sum_{0 \leq n \leq V} \frac{1}{(n+\alpha)^{\sigma+iv}}\right|^2 \,dv\right\}^{1/2}
+ (\log{V}) V^{-\sigma}
+ C V^{-\sigma} \nonumber
\end{align}
by using the Cauchy--Schwarz inequality, where $C$ is a positive absolute constant. 
Furthermore, we apply Hilbert's inequality \cite[Corollary 2]{MontgomeryVaughan1974} to see that
\begin{gather*}
\int_{0}^{V} \left|\sum_{0 \leq n \leq V} \frac{1}{(n+\alpha)^{\sigma+iv}}\right|^2 \,dv
\leq \sum_{0 \leq n \leq V} \frac{1}{(n+\alpha)^{2\sigma}} (V+3\pi \delta_n^{-1}), 
\end{gather*}
where $\delta_n=\min_{m \neq n} |\log(n+\alpha)-\log(m+\alpha)|$. 
We have 
\begin{gather*}
\delta_n
\geq \log\left(\frac{n+1+\alpha}{n+\alpha}\right)
\geq \frac{1}{n+2}
\gg \frac{1}{V}
\end{gather*}
for any $0 \leq n \leq V$. 
Hence we derive
\begin{align*}
\frac{1}{V} \int_{0}^{V} \left|\sum_{0 \leq n \leq V} \frac{1}{(n+\alpha)^{\sigma+iv}}\right|^2 \,dv 
\ll \sum_{0 \leq n \leq V} \frac{1}{(n+\alpha)^{2\sigma}}. 
\end{align*}
Furthermore, the last sum is estimated as
\begin{gather*}
\sum_{0 \leq n \leq V} \frac{1}{(n+\alpha)^{2\sigma}}
\leq \alpha^{-1}
+ \sum_{n=1}^{\infty} n^{-2\sigma}
\ll \frac{1}{2\sigma-1},  
\end{gather*}
where the implied constant depends only on $\alpha$. 
By \eqref{eq:04170044}, we obtain
\begin{gather}\label{eq:04160052}
\frac{1}{V} \int_{2\pi}^{V} |\zeta(\sigma+iv,\alpha)| \,dv
\ll \frac{1}{\sqrt{2\sigma-1}}
\end{gather}
since $(\log{V}) V^{-\sigma}+V^{-\sigma} \ll 1$ and $1/\sqrt{2\sigma-1} \gg1$. 
Furthermore, \eqref{eq:04160052} implies
\begin{gather}\label{eq:04160053}
\frac{1}{V} \int_{-V}^{-2\pi} |\zeta(\sigma+iv,\alpha)| \,dv
\ll \frac{1}{\sqrt{2\sigma-1}}
\end{gather}
by the identity $\zeta(\sigma-iv,\alpha)=\overline{\zeta(\sigma+iv,\alpha)}$. 
Then, we recall that $\zeta(s,\alpha)$ has a simple pole at $s=1$. 
Thus we deduce
\begin{gather}\label{eq:04160057}
\frac{1}{V} \int_{-2\pi}^{2\pi} |\zeta(\sigma+iv,\alpha)| \,dv
\ll \sup_{|v| \leq 2\pi} |\zeta(\sigma+iv,\alpha)|
\ll \frac{1}{1-\sigma}, 
\end{gather}
where the implied constant depends only on $\alpha$. 
Combining \eqref{eq:04160052}, \eqref{eq:04160053}, and \eqref{eq:04160057}, we obtain \eqref{eq:04170059}. 
Then, we prove the first part of the lemma. 
Let $s=\sigma+it$ be a complex number with $1/2<\sigma<1$. 
We have 
\begin{gather*}
\frac{1}{T} \int_{0}^{T} |\zeta(s+i \tau, \alpha)| \,d \tau
= \frac{1}{T} \int_{t}^{T+t} |\zeta(\sigma+i \tau, \alpha)| \,d \tau 
\leq \frac{1}{T} \int_{-(T+|t|)}^{T+|t|} |\zeta(\sigma+i \tau, \alpha)| \,d \tau
\end{gather*}
since $T+t \leq T+|t|$ and $t \geq -(T+|t|)$. 
Applying \eqref{eq:04170059} with $V=T+|t|$, we obtain
\begin{align*}
\frac{1}{T} \int_{0}^{T} |\zeta(s+i \tau, \alpha)| \,d \tau 
&\leq \left(1+\frac{|t|}{T}\right) \frac{1}{T+|t|} \int_{-(T+|t|)}^{T+|t|} |\zeta(\sigma+i \tau, \alpha)| \,d \tau \\
&\ll \frac{1}{\sqrt{2\sigma-1}} \left(1+\frac{|t|}{T}\right)
+ \frac{1}{1-\sigma} \left(1+\frac{|t|}{T}\right)
\end{align*}
as desired. 
The second part is proved as follows. 
By the Cauchy--Schwarz inequality, we have
\begin{gather*}
\mathbf{E} \left[|\zeta(s,\mathbb{X}_\alpha)|\right]^2
\leq \mathbf{E} \left[|\zeta(s,\mathbb{X}_\alpha)|^2\right]
= \sum_{m=0}^{\infty} \sum_{n=0}^{\infty} 
\frac{\mathbf{E} [\mathbb{X}_\alpha(m) \overline{\mathbb{X}_\alpha(n)}]}{(m+\alpha)^s (n+\alpha)^{\overline{s}}}. 
\end{gather*}
Lemma \ref{lem:2.1} yields that $\mathbf{E} [\mathbb{X}_\alpha(m) \overline{\mathbb{X}_\alpha(n)}]=0$ for $m \neq n$. 
As a result, we deduce
\begin{gather*}
\mathbf{E} \left[|\zeta(s,\mathbb{X}_\alpha)|\right]
\leq \left\{\sum_{n=0}^{\infty} \frac{1}{(n+\alpha)^{2\sigma}} \right\}^{1/2}
\ll \frac{1}{\sqrt{2\sigma-1}}, 
\end{gather*}
where the implied constant depends only on $\alpha$. 
Hence the proof is completed. 
\end{proof}

Let $d$ be the distance of $H(D)$ as in \eqref{eq:04131758}. 
Since $x/(1+x) \leq \min\{x,1\}$ for $x>0$, we obtain the inequality
\begin{gather}\label{eq:04232225}
d(f,g)
\leq \sum_{\nu=1}^{M} 2^{-\nu} d_\nu(f,g) + \sum_{\nu=M+1}^{\infty} 2^{-\nu}
\leq d_M(f,g) + 2^{-M}
\end{gather}
for any $f,g \in H(D)$ and any $M \geq1$. 
Then, applying Lemmas \ref{lem:3.4} and \ref{lem:3.5}, we prove the following result. 

\begin{proposition}\label{prop:3.6}
Let $\alpha$ be an algebraic number with $0<\alpha \leq1$. 
Then there exists an absolute constant $c>0$ such that
\begin{gather*}
\frac{1}{T} \int_{0}^{T} d(\zeta(s+i \tau,\alpha), Z_N(s+i \tau,\alpha)) \,d \tau
\ll \exp\left(-c\sqrt{\log{N}}\right) 
+ \frac{N(\log{N})^3}{T} 
\end{gather*}
for any $T \geq3$ and sufficiently large $N$. 
Furthermore, we have 
\begin{gather*}
\mathbf{E} \left[ d(\zeta(s,\mathbb{X}_\alpha), Z_N(s,\mathbb{X}_\alpha)) \right]
\ll \exp\left(-c\sqrt{\log{N}}\right) 
\end{gather*}
for sufficiently large $N$. 
Here, the implied constants depend only on $\alpha$. 
\end{proposition}

\begin{proof}
Let $M$ be a positive integer chosen later. 
By \eqref{eq:04232225}, we have 
\begin{align}\label{eq:04150024}
&\frac{1}{T} \int_{0}^{T} d(\zeta(s+i \tau,\alpha), Z_N(s+i \tau,\alpha)) \,d \tau \\
&\leq \frac{1}{T} \int_{0}^{T} \sup_{s \in K_M} \left|\zeta(s+i \tau,\alpha)-Z_N(s+i \tau,\alpha)\right| \,d \tau
+ 2^{-M}. \nonumber
\end{align}
Let $s \in K_M$ and $\tau \in \mathbb{R}$. 
Then Cauchy's integral formula is available to obtain 
\begin{gather*}
\zeta(s+i \tau,\alpha)-Z_N(s+i \tau,\alpha)
= \frac{1}{2\pi i} \oint_{\partial K_{M+1}} \frac{\zeta(z+i \tau,\alpha)-Z_N(z+i \tau,\alpha)}{z-s} \,dz. 
\end{gather*}
By the definition of $\{K_\nu\}$, we see that $|z-s| \geq (5M(M+1))^{-1}$ for any $s \in K_M$ and any $z \in \partial K_{M+1}$. 
Therefore we obtain
\begin{align}\label{eq:04150021}
&\sup_{s \in K_M} \left|\zeta(s+i \tau,\alpha)-Z_N(s+i \tau,\alpha)\right| \\
&\leq \frac{5}{2\pi} M(M+1) 
\oint_{\partial K_{M+1}} \left|\zeta(z+i \tau,\alpha)-Z_N(z+i \tau,\alpha)\right| \,|dz|. \nonumber
\end{align}
For any $z \in \partial K_{M+1}$, we have $1/2<\RE(z+i \tau)<1$. 
Then, we apply Lemma \ref{lem:3.4} to derive the estimate
\begin{align*}
&\left|\zeta(z+i \tau,\alpha)-Z_N(z+i \tau,\alpha)\right| \\
&\leq \frac{1}{2\pi} \int_{-\delta-i \infty}^{-\delta+i \infty} |\zeta(z+w+i \tau,\alpha)| |\widehat{\phi}(w)| N^{\RE(w)} \,|dw| 
+ |\widehat{\phi}(1-(z+i \tau))| N^{1-\RE(z)} \\
&\ll N^{-\delta} \int_{-\delta-i \infty}^{-\delta+i \infty} |\zeta(z+w+i \tau,\alpha)| |\widehat{\phi}(w)| \,|dw| 
+ N^{1/2} |\widehat{\phi}(1-z-i \tau)|
\end{align*}
for any $z \in \partial K_{M+1}$, where $\delta$ satisfies $0<\delta<\RE(z)-1/2$. 
Inserting this estimate to \eqref{eq:04150021}, we obtain
\begin{align*}
&\sup_{s \in K_M} \left|\zeta(s+i \tau,\alpha)-Z_N(s+i \tau,\alpha)\right| \\
&\ll M^2 N^{-\delta} \oint_{\partial K_{M+1}} \int_{-\delta-i \infty}^{-\delta+i \infty} |\zeta(z+w+i \tau,\alpha)| |\widehat{\phi}(w)| \,|dw| |dz|\\
&\qquad
+ M^2 N \oint_{\partial K_{M+1}} |\widehat{\phi}(1-z-i \tau)| \,|dz|. 
\end{align*}
Using \eqref{eq:04150024} and changing the orders of integrals, we arrive at
\begin{align}\label{eq:04150239}
&\frac{1}{T} \int_{0}^{T} d(\zeta(s+i \tau,\alpha), Z_N(s+i \tau,\alpha)) \,d \tau \\
&\ll M^2 N^{-\delta} \oint_{\partial K_{M+1}} \int_{-\delta-i \infty}^{-\delta+i \infty} 
\left(\frac{1}{T} \int_{0}^{T} |\zeta(z+w+i \tau,\alpha)| \,d \tau\right) |\widehat{\phi}(w)| \,|dw| |dz| \nonumber\\
&\qquad
+ M^2 N \oint_{\partial K_{M+1}} \left(\frac{1}{T} \int_{0}^{T} |\widehat{\phi}(1-z-i \tau)| \,d \tau\right) \,|dz|
+ 2^{-M}. \nonumber
\end{align}
Let $z \in \partial K_{M+1}$, and put $\delta=(10(M+1))^{-1}$ so that $0<\delta<\RE(z)-1/2$ is satisfied. 
Furthermore, we have $1/2+(10(M+1))^{-1} \leq \RE(z+w) \leq 1-(5(M+1))^{-1}$ on the line $\RE(w)=-\delta$. 
By Lemma \ref{lem:3.5}, we obtain
\begin{align}\label{eq:04150240}
\frac{1}{T} \int_{0}^{T} |\zeta(z+w+i \tau,\alpha)| \,d \tau
&\ll M \left(1+\frac{|\IM(z+w)|}{T}\right) \\
&\ll M^2 (|\IM(w)|+1) \nonumber
\end{align}
for any $T \geq3$, where the implied constant depends only on $\alpha$. 
On the other hand, we have $0 \leq \RE(1-z-i \tau) \leq 1/2$. 
Furthermore, $|\IM(1-z-i \tau)| \geq \tau/2 \geq1$ is satisfied if $\tau \in [2(M+1), T]$. 
Hence we deduce from \eqref{eq:04150149} that $|\widehat{\phi}(1-z-i \tau)|\ll \tau^{-2}$ for any $\tau \in [2(M+1), T]$, where the implied constant is absolute. 
Hence the integral of $|\widehat{\phi}(1-z-i \tau)|$ is evaluated as
\begin{align*}
\frac{1}{T} \int_{0}^{T} |\widehat{\phi}(1-z-i \tau)| \,d \tau
&\ll \frac{1}{T} \int_{2(M+1)}^{T} \tau^{-2} \,d \tau
+ \frac{1}{T} \int_{0}^{2(M+1)} |\widehat{\phi}(1-z-i \tau)| \,d \tau \\
&\leq \frac{1}{T} 
+ \frac{2(M+1)}{T} \sup_{0\leq \tau \leq2(M+1)} |\widehat{\phi}(1-z-i \tau)| 
\end{align*}
for any $T \geq 2(M+1)$. 
Note that the same estimate is valid even if $T<2(M+1)$, since we have 
\begin{gather*}
\frac{1}{T} \int_{0}^{T} |\widehat{\phi}(1-z-i \tau)| \,d \tau
\leq \frac{2(M+1)}{T} \sup_{0\leq \tau \leq2(M+1)} |\widehat{\phi}(1-z-i \tau)| 
\end{gather*}
for any $T<2(M+1)$. 
By the definition of the Mellin transform, we have 
\begin{align*}
|\widehat{\phi}(1-z-i \tau)|
&\leq \int_{0}^{\infty} \phi(x) x^{1-\RE(z)} \,\frac{dx}{x} \\
&\leq \int_{1}^{\infty} \phi(x) x^{1/2} \,\frac{dx}{x}
+ \int_{0}^{1} \phi(x) x^{1/(5(M+1))} \,\frac{dx}{x} \\
&\leq \widehat{\phi}\left(\frac{1}{2}\right)+\widehat{\phi}\left(\frac{1}{5(M+1)}\right)
\end{align*}
for any $z \in \partial K_{M+1}$ and $\tau \in \mathbb{R}$. 
Since $\widehat{\phi}(w)$ has a simple pole at $w=0$, we obtain that $\widehat{\phi}(1/(5(M+1))) \ll M$. 
Hence $|\widehat{\phi}(1-z-i \tau)| \ll M$ follows. 
From the above, we arrive at
\begin{gather}\label{eq:04150241}
\frac{1}{T} \int_{0}^{T} |\widehat{\phi}(1-z-i \tau)| \,d \tau
\ll \frac{M^2}{T}
\end{gather}
for any $T \geq3$, where the implied constant is absolute. 
Inserting \eqref{eq:04150240} and \eqref{eq:04150241} to \eqref{eq:04150239}, we deduce that
\begin{align*}
&\frac{1}{T} \int_{0}^{T} d(\zeta(s+i \tau,\alpha), Z_N(s+i \tau,\alpha)) \,d \tau \\
&\ll M^{4} N^{-\delta} \oint_{\partial K_{M+1}} \,|dz|
\int_{-\delta-i \infty}^{-\delta+i \infty} (|\IM(w)|+1) |\widehat{\phi}(w)| \,|dw| \\
&\qquad\quad
+ \frac{M^4 N}{T} \oint_{\partial K_{M+1}} \,|dz|
+ 2^{-M} \\ 
&\ll M^{5} N^{-\delta} 
\int_{-\delta-i \infty}^{-\delta+i \infty} (|\IM(w)|+1) |\widehat{\phi}(w)| \,|dw| 
+ \frac{M^5 N}{T} 
+ 2^{-M},  
\end{align*}
where the last line is derived from  $\oint_{\partial K_{M+1}} \,|dz| \ll M$. 
Recall that $\delta=(10(M+1))^{-1}$. 
Hence we have $-1/2 \leq -\delta \leq0$. 
Applying \eqref{eq:04150149} with $k=3$, we obtain
\begin{align*}
\int_{-\delta-i \infty}^{-\delta+i \infty} (|\IM(w)|+1) |\widehat{\phi}(w)| \,|dw| 
&\ll \int_{|v| \geq1} |v|^{-2} \,dv
+ \int_{|v| \leq1} |\widehat{\phi}(-\delta+iv)| \,dv \\
&\ll M
\end{align*}
since $\widehat{\phi}(-\delta+iv) \ll M$ for any $|v| \leq1$, which follows from the fact that $\widehat{\phi}(w)$ has a simple pole at $w=0$.   
Here, we choose the positive integer $M$ as $M= \lfloor\sqrt{\log{N}}\rfloor$ for $N \geq3$. 
We finally obtain
\begin{align*}
\frac{1}{T} \int_{0}^{T} d(\zeta(s+i \tau,\alpha), Z_N(s+i \tau,\alpha)) \,d \tau 
&\ll M^6 N^{-\delta} 
+ \frac{M^5 N}{T} 
+ 2^{-M} \\
&\ll \exp\left(-c\sqrt{\log{N}}\right) 
+ \frac{N(\log{N})^3}{T} 
\end{align*}
for sufficiently large $N$, where $c$ is an absolute constant. 
Then, we prove the second part of the proposition. 
In a similar way that we derive \eqref{eq:04150239}, we also obtain 
\begin{align*}
&\mathbf{E} \left[ d(\zeta(s,\mathbb{X}_\alpha), Z_N(s,\mathbb{X}_\alpha)) \right] \\
&\ll M^2 N^{-\delta} \oint_{\partial K_{M+1}} \int_{-\delta-i \infty}^{-\delta+i \infty} 
\mathbf{E} \left[|\zeta(s,\mathbb{X}_\alpha)|\right] |\widehat{\phi}(w)| \,|dw| \,|dz| 
+ 2^{-M}
\end{align*}
by noting the disappearance of the term $\widehat{\phi}(1-s) N^{1-s}$ in the formula for $\zeta(s,\mathbb{X}_\alpha)$ of Lemma \ref{lem:3.4}. 
Applying Lemma \ref{lem:3.5}, we have
\begin{align*}
\mathbf{E} \left[ d(\zeta(s,\mathbb{X}_\alpha), Z_N(s,\mathbb{X}_\alpha)) \right] 
&\ll M^3 N^{-\delta} \oint_{\partial K_{M+1}} \,|dz| \int_{-\delta-i \infty}^{-\delta+i \infty} |\widehat{\phi}(w)| \,|dw|  
+ 2^{-M} \\
&\ll M^5 N^{-\delta} 
+ 2^{-M} \\
&\ll \exp\left(-c\sqrt{\log{N}}\right)
\end{align*}
for sufficiently large $N$. 
From the above, we obtain the conclusion. 
\end{proof}

\subsection{Proof of Theorem \ref{thm:3.1}}\label{sec:3.3}
Let $S$ be a metric space with distance $d$. 
We say that $F: S \to \mathbb{C}$ is a Lipschitz function if there exists a constant $\mathcal{L}>0$ such that
\begin{gather*}
|F(x)-F(y)|
\leq \mathcal{L} d(x,y)
\end{gather*}
for any $x,y \in S$. 
The constant $\mathcal{L}$ is called a Lipschitz constant for $F$. 
By definition, we know that any Lipschitz function is continuous. 
Denote by $\mathcal{B}(S)$ the Borel algebra of $S$ with the topology induced from $d$. 

\begin{lemma}[Portmanteau theorem]\label{lem:3.7}
Let $\{P_n\}$ be a sequence of probability measures on $(S, \mathcal{B}(S))$. 
Let $Q$ be a probability measure on $(S, \mathcal{B}(S))$. 
Then the followings are equivalent.  
\begin{itemize}
\item[$(\mathrm{i})$]
$P_n$ converges weakly to $Q$ as $n \to\infty$.  
\item[$(\mathrm{ii})$]
For any bounded Lipschitz function $F:S \to \mathbb{C}$, we have
\begin{gather*}
\lim_{n \to\infty} \int_{S} F \,dP_n
= \int_{S} F \,dQ. 
\end{gather*}
\item[$(\mathrm{iii})$]
For any open set $A$ of $S$, we have 
\begin{gather*}
\liminf_{n \to\infty} P_n(A) 
\geq Q(A). 
\end{gather*}
\end{itemize}
\end{lemma}

\begin{proof}
See \cite[Theorem 13.16]{Klenke2020} for a proof. 
\end{proof}

From the above preparations, we finally prove Theorem \ref{thm:3.1}. 
By the definitions of $P_{\alpha,T}$ and $Q_\alpha$, the integrals with respect to these probability measures are 
\begin{gather*}
\int_{H(D)} F \,dP_{\alpha,T}
= \frac{1}{T} \int_{0}^{T} F(\zeta(s+i \tau, \alpha)) \,d \tau, \\
\int_{H(D)} F \,dQ_{\alpha}
= \mathbf{E} \left[ F(\zeta(s, \mathbb{X}_\alpha)) \right]
\end{gather*}
for any measurable function $F: H(D) \to \mathbb{C}$. 

\begin{proof}[Proof of Theorem \ref{thm:3.1}]
Let $F: H(D) \to \mathbb{C}$ be a bounded Lipschitz function with a Lipschitz constant $\mathcal{L}$. 
Then we derive by Proposition \ref{prop:3.6} that
\begin{align}\label{eq:04171827}
&\left|\frac{1}{T} \int_{0}^{T} F(\zeta(s+i \tau, \alpha)) \,d \tau
- \frac{1}{T} \int_{0}^{T} F(Z_N(s+i \tau, \alpha)) \,d \tau\right| \\
&\leq \frac{\mathcal{L}}{T} \int_{0}^{T} d(\zeta(s+i \tau,\alpha), Z_N(s+i \tau,\alpha)) \,d \tau \nonumber\\
&\leq \mathcal{L}_\alpha \left\{\exp\left(-c\sqrt{\log{N}}\right) \nonumber
+ \frac{N(\log{N})^3}{T}\right\}
\end{align}
for any $T \geq3$ and sufficiently large $N$, where $\mathcal{L}_\alpha>0$ is a constant depending only on $\alpha$. 
Analogously, we obtain
\begin{gather}\label{eq:04171828}
\Big|\mathbf{E} \left[ F(\zeta(s, \mathbb{X}_\alpha)) \right]-\mathbf{E} \left[ F(Z_N(s, \mathbb{X}_\alpha)) \right]\Big|
\leq \mathcal{L}_\alpha \exp\left(-c\sqrt{\log{N}}\right) 
\end{gather}
for sufficiently large $N$. 
Since the function $\phi$ is compactly supported, there exists an integer $k=k(\alpha,N) \geq1$ such that $\phi((n+\alpha)/N)=0$ for any integer $n \geq k$.  
Let $\psi_N: \mathcal{T}^{k} \to H(D)$ be a continuous map defined as
\begin{gather*}
\psi_N(\underline{\gamma})
= \sum_{n=0}^{k-1} \frac{\gamma_n}{(n+\alpha)^s} \phi \left(\frac{n+\alpha}{N}\right)
\end{gather*}
for $\underline{\gamma}=(\gamma_0,\ldots,\gamma_{k-1}) \in \mathcal{T}^k$. 
Let $\mu_{\alpha,T}$ and $\nu_\alpha$ denote the probability measures on $(\mathcal{T}^k, \mathcal{B}(\mathcal{T}^k))$ as in Proposition \ref{prop:3.3}, where we put $n_j=j-1$ for $j=1,\ldots,k$. 
Then we obtain
\begin{gather*}
\frac{1}{T} \int_{0}^{T} F(Z_N(s+i \tau, \alpha)) \,d \tau
= \int_{\mathcal{T}^k} (F \circ \psi_N) \,d \mu_{\alpha,T}, \\
\mathbf{E} \left[ F(Z_N(s, \mathbb{X}_\alpha)) \right]
= \int_{\mathcal{T}^k} (F \circ \psi_N) \,d \nu_{\alpha}. 
\end{gather*}
Note that $F \circ \psi_N$ is a bounded continuous function on $\mathcal{T}^k$ for any $N \geq1$. 
Hence, we derive the limit formula
\begin{gather}\label{eq:04171829}
\lim_{T \to\infty} \frac{1}{T} \int_{0}^{T} F(Z_N(s+i \tau, \alpha)) \,d \tau
= \mathbf{E} \left[ F(Z_N(s, \mathbb{X}_\alpha)) \right]
\end{gather}
for any $N \geq1$ by Proposition \ref{prop:3.3}. 
Let $\epsilon$ be any positive real number. 
By \eqref{eq:04171827} and \eqref{eq:04171828}, there exists an integer $N_0=N_0(\epsilon,\alpha) \geq3$ such that
\begin{gather*}
\left|\frac{1}{T} \int_{0}^{T} F(\zeta(s+i \tau, \alpha)) \,d \tau
- \frac{1}{T} \int_{0}^{T} F(Z_{N_0}(s+i \tau, \alpha)) \,d \tau\right|
< \frac{\epsilon}{3}, \\
\Big|\mathbf{E} \left[ F(\zeta(s, \mathbb{X}_\alpha)) \right]
- \mathbf{E} \left[ F(Z_{N_0}(s, \mathbb{X}_\alpha)) \right]\Big|
< \frac{\epsilon}{3}
\end{gather*}
for any $T \geq T_1:=(6/\epsilon) \mathcal{L}_\alpha N_0 (\log{N}_0)^3$. 
Furthermore, we derive by \eqref{eq:04171829} that there exists a real number $T_2=T_2(\epsilon,\alpha) \geq3$ such that
\begin{gather*}
\left|\frac{1}{T} \int_{0}^{T} F(Z_{N_0}(s+i \tau, \alpha)) \,d \tau
- \mathbf{E} \left[ F(Z_{N_0}(s, \mathbb{X}_\alpha)) \right]\right|
< \frac{\epsilon}{3}
\end{gather*}
for any $T \geq T_2$. 
As a result, we have 
\begin{gather*}
\left|\frac{1}{T} \int_{0}^{T} F(\zeta(s+i \tau, \alpha)) \,d \tau
- \mathbf{E} \left[ F(\zeta(s, \mathbb{X}_\alpha)) \right]\right|
< \epsilon
\end{gather*}
for any $T \geq \max\{T_1,T_2\}$. 
By Lemma \ref{lem:3.7}, the proof is completed. 
\end{proof}

\section{Covering theorem for the space $A^2(U)$}\label{sec:4}
Let $U$ be a bounded domain with $\overline{U} \subset D$. 
Define the inner product and norm as
\begin{gather*}
\left\langle f,g \right\rangle
= \iint_{U} f(s) \overline{g(s)} \,d \sigma dt
\quad\text{and}\quad
\|f\|=\sqrt{\left\langle f,f \right\rangle}
\end{gather*}
for measurable functions $f, g: U \to \mathbb{C}$. 
The Bergman space $A^2(U)$ is defined as the space of all analytic functions $f:U \to \mathbb{C}$ such that $\|f\|<\infty$. 
Then it is a complex Hilbert space with the inner product described above. 
Throughout this section, we suppose that the boundary $\partial U$ is a Jordan curve. 
Then the subspace
\begin{gather}\label{eq:04191326}
\mathcal{P}
= \left\{a_0+a_1 s+\cdots+a_N s^N ~\middle|~ \text{$a_k \in \mathbb{C}$ for $0 \leq k \leq N$ with $N \geq0$}\right\}
\end{gather}
is dense in $A^2(U)$. 
See \cite{Farrell1934} for a proof. 
Let $X$ be any subset of $A^2(U)$. 
For every $\epsilon>0$, we denote by $X^{(\epsilon)}$ the $\epsilon$-neighborhood of $X$ defined as
\begin{gather*}
X^{(\epsilon)}
= \left\{f \in A^2(U) ~\middle|~ \text{$\exists g \in X$ such that $\|f-g\|<\epsilon$}\right\}. 
\end{gather*}
Furthermore, we define
\begin{gather*}
\Gamma(\alpha,N)
= \left\{ \sum_{n=0}^{N} \frac{\gamma_n}{(n+\alpha)^s} ~\middle|~ 
\text{$|\gamma_n|=1$ for $0 \leq n \leq N$}  \right\}
\end{gather*}
for $0<\alpha \leq1$ and $N \geq0$. 
In this section, we prove the following covering theorem for the space $A^2(U)$. 

\begin{theorem}\label{thm:4.1}
Let $0<c<1$. 
Then, for every $\epsilon>0$, there exists a positive real number $\rho$ such that
\begin{gather*}
A^2(U)
= \bigcup_{N_0\geq0} \bigcap_{N> N_0} \bigcap_{\alpha \in \mathcal{A}_\rho(c)} \Gamma(\alpha,N)^{(\epsilon)}, 
\end{gather*}
where $\rho$ depends only on $c,\epsilon,U$. 
\end{theorem}

\subsection{Preliminary lemmas}\label{sec:4.1}
Let $H$ be a complex Hilbert space. 
Recall that any continuous linear functional $f: H \to \mathbb{C}$ is represented as $f(x)=\left\langle x,y \right\rangle$ with some $y \in H$ by the Riesz representation theorem. 
We say that a subset $K \subset H$ is convex if $t x+(1-t)y \in K$ for any $x,y \in K$ and any $0 \leq t \leq1$. 

\begin{lemma}\label{lem:4.2}
Let $H$ be a complex Hilbert space. 
Let $K$ be any closed convex subset of $H$, and suppose $x \in H \setminus K$. 
Then there exist an element $y \in H$ and a constant $\eta \in \mathbb{R}$ such that 
\begin{gather*}
\RE \left\langle z,y \right\rangle
\leq \eta
< \RE \left\langle x,y \right\rangle
\end{gather*}
for all $z \in K$. 
\end{lemma}

\begin{proof}
This is a special case of the Hahn--Banach separation theorem, which holds for every locally convex vector space. 
See \cite[Theorem 8.73]{EinsiedlerWard2017} for a proof. 
\end{proof}

For a subspace $L$ of $H$, we denote by $L^\perp$ the orthogonal complement, that is, 
\begin{gather*}
L^\perp
= \{x \in H \mid \text{$\forall y \in L$, $\left\langle x,y \right\rangle=0$} \}. 
\end{gather*}
If $L$ is a closed subspace, then every element $x \in H$ has a unique representation $x=y+z$ such that $y \in L$ and $z \in L^\perp$. 
Thus we obtain $H= L \oplus L^\perp$. 

\begin{lemma}\label{lem:4.3}
Let $H$ be a complex Hilbert space. 
Then a subspace $L$ is dense in $H$ if and only if $L^\perp=\{0\}$. 
\end{lemma}

\begin{proof}
Note that $M= \overline{L}$ is a closed subspace of $H$. 
Then the result follows from the decomposition $H= M \oplus M^\perp$. 
\end{proof}

\begin{lemma}\label{lem:4.4}
Let $H$ be a complex Hilbert space, and take $x_1,\ldots,x_n \in H$ arbitrarily. 
Let $\beta_1,\ldots,\beta_n$ be complex numbers with $|\beta_j| \leq1$ for $1 \leq j \leq n$. 
Then we have 
\begin{gather*}
\left\| \sum_{j=1}^{n} \beta_j x_j 
- \sum_{j=1}^{n} \gamma_j x_j \right\|^2
\leq 4 \sum_{j=1}^{n} \|x_j\|^2
\end{gather*}
with some complex numbers $\gamma_1,\ldots,\gamma_n$ with $|\gamma_j|=1$ for $1 \leq j \leq n$.
\end{lemma}

\begin{proof}
See \cite[Lemma 6.1.15]{Laurincikas1996a} for a proof. 
\end{proof}

Let $F(z)$ be an entire function. 
We say that $F(z)$ is of exponential type if 
\begin{gather*}
\limsup_{r \to\infty} \frac{\log|F(re^{i \theta})|}{r}
< \infty
\end{gather*}
uniformly in $\theta \in \mathbb{R}$. 
The following lemmas are also used to prove Theorem \ref{thm:4.1}. 

\begin{lemma}\label{lem:4.5}
Let $F(z)$ be an entire function of exponential type. 
Let $\{\lambda_m\}$ be a sequence of complex numbers. 
If there exist positive real numbers $\alpha,\beta,\delta$ such that 
\begin{itemize}
\item[$(\mathrm{a})$]
$\displaystyle{\limsup_{y \to\infty} \frac{\log|F(\pm iy)|}{y} \leq \alpha}$, 
\item[$(\mathrm{b})$]
$\displaystyle{|\lambda_m-\lambda_n| \geq \delta |m-n|}$, 
\item[$(\mathrm{c})$]
$\displaystyle{\lim_{m \to\infty} \frac{\lambda_m}{m}=\beta}$, 
\item[$(\mathrm{d})$]
$\displaystyle{\alpha \beta <\pi}$, 
\end{itemize}
then we have 
\begin{gather*}
\limsup_{m \to\infty} \frac{\log|F(\lambda_m)|}{|\lambda_m|}
= \limsup_{r \to\infty} \frac{\log|F(r)|}{r}. 
\end{gather*}
\end{lemma}

\begin{proof}
See \cite[Theorem 6.4.12]{Laurincikas1996a} for a proof. 
\end{proof}

\begin{lemma}\label{lem:4.6}
Let $\mu$ be a comple measure on $(\mathbb{C}, \mathcal{B}(\mathbb{C}))$ whose support is compact and contained in the half-plane $\sigma>\sigma_0$. 
If the function $F(z)$ defined by
\begin{gather*}
F(z)
= \int_{\mathbb{C}} e^{zs} \,d \mu(s)
\end{gather*}
for $z \in \mathbb{C}$ does not vanish everywhere, then we have 
\begin{gather*}
\limsup_{r \to\infty} \frac{\log|F(r)|}{r}
> \sigma_0. 
\end{gather*}
\end{lemma}

\begin{proof}
See \cite[Lemma 6.4.10]{Laurincikas1996a} for a proof. 
\end{proof}

\subsection{Proof of Theorem \ref{thm:4.1}}\label{sec:4.2}
Before proving Theorem \ref{thm:4.1}, we show two auxiliary results by applying the lemmas in Section \ref{sec:4.1}. 
The first result is a denseness result for the space $A^2(U)$. 
Note that a similar result is also used in Bagchi's method for the proof of universality in \cite{Bagchi1981}. 

\begin{proposition}\label{prop:4.7}
Let $0<c<1$. 
Then the set 
\begin{gather*}
B(c,M)
= \left\{ \sum_{M<n \leq N} \frac{\beta_n}{(n+c)^s} ~\middle|~ 
\text{$|\beta_n| \leq1$ for $M<n \leq N$ with $N \geq M+1$}  \right\}
\end{gather*}
is dense in $A^2(U)$ for any $M \geq0$. 
\end{proposition}

\begin{proof}
Suppose $\overline{B(c,M)} \neq A^2(U)$ for some $M \geq0$. 
Then we can take a function $f$ in $A^2(U)$ such that $f \notin \overline{B(c,M)}$. 
By definition, the set $B(c,M)$ is convex. 
Thus the closure $\overline{B(c,M)}$ is also a convex subset, which is closed in $A^2(U)$. 
By Lemma \ref{lem:4.2}, there exist a function $g \in A^2(U)$ and a constant $\eta \in \mathbb{R}$ such that 
\begin{gather}\label{eq:04191249}
\RE \left\langle h,g \right\rangle
\leq \eta
< \RE \left\langle f,g \right\rangle
\end{gather}
for all $h \in \overline{B(c,M)}$. 
Put 
\begin{gather*}
h_N(s)
= \sum_{M<n \leq N} 
\frac{|\left\langle (n+c)^{-s}, g(s) \right\rangle|}{\left\langle (n+c)^{-s}, g(s) \right\rangle} 
\frac{1}{(n+c)^s}
\end{gather*}
for $N \geq M+1$. 
It is obviously an element of the set $\overline{B(c,M)}$. 
By \eqref{eq:04191249}, we have 
\begin{gather*}
\RE \left\langle h_N, g \right\rangle
= \sum_{M<n \leq N} \left|\left\langle (n+c)^{-s}, g(s) \right\rangle\right|
< \RE \left\langle f,g \right\rangle
\end{gather*}
for any $N \geq M+1$. 
Therefore, we find that 
\begin{gather}\label{eq:04200421}
\sum_{n=0}^{\infty} \left|\left\langle (n+c)^{-s}, g(s) \right\rangle\right|
< \infty 
\end{gather}
must be satisfied. 
Using the function $g$, we define $F_g(z) = \left\langle e^{-zs}, g(s) \right\rangle$ for $z \in \mathbb{C}$. 
Put $\alpha=\max\{ |s| \mid s \in \overline{U}\}$. 
Then the Cauchy--Schwarz inequality yields that
\begin{align}\label{eq:04200506}
|F_g(r e^{i \theta})| 
&\leq \left\{\int_{U} |\exp(-r e^{i \theta}s)|^2 \,d \sigma dt\right\}^{1/2} 
\left\{\int_{U} |g(s)|^2 \,d \sigma dt\right\}^{1/2} \\
&\ll \exp(\alpha r) \nonumber
\end{align}
uniformly in $\theta \in \mathbb{R}$. 
Hence $F_g(z)$ is an entire function of exponential type. 
Furthermore, it does not vanish everywhere. 
Indeed, if $F_g(z)=0$ for any $z \in \mathbb{C}$, then 
\begin{gather*}
F_g(0)
= \left\langle 1, g(s) \right\rangle
=0
\quad\text{and}\quad
\frac{d^k}{dz^k} F_g(z) \bigg|_{z=0}
= (-1)^k \langle s^k, g(s) \rangle
= 0
\end{gather*}
would hold for all $k \geq1$. 
These imply that $g \in \mathcal{P}^\perp$, where $\mathcal{P}$ is the subspace of $A^2(U)$ as in \eqref{eq:04191326}. 
However, since $\mathcal{P}$ is dense in $A^2(U)$, we get $g=0$ by Lemma \ref{lem:4.3}. 
This contradicts inequality \eqref{eq:04191249}. 
Then, we prove 
\begin{gather}\label{eq:04200509}
\limsup_{r \to\infty} \frac{\log|F_g(r)|}{r}
\leq -1 
\end{gather}
by applying Lemma \ref{lem:4.5}. 
We put $\beta=\pi/(2\alpha)$ so that condition $(\mathrm{d})$ of Lemma \ref{lem:4.5} is satisfied. 
By \eqref{eq:04200506}, we see that condition $(\mathrm{a})$ is also satisfied. 
Define $A_\beta$ as the set of all integers $m \geq1$ such that $|F_g(r)|<e^{-r}$ holds with some $r \in \mathbb{R}$ satisfying $m \beta <r<(m+1/2)\beta$. 
Let 
\begin{gather*}
C_\beta(m)
= \left\{ n \in \mathbb{Z}_{\geq0} ~\middle|~ m\beta <\log(n+c)<\left(m+\frac{1}{2}\right)\beta \right\} 
\end{gather*}
for $m \geq1$. 
By the definitions of $A_\beta$ and $C_\beta(m)$, we have 
\begin{gather*}
\sum_{n=0}^{\infty} |F_g(\log(n+c))|
\geq \sum_{m \notin A_\beta} \sum_{n \in C_\beta(m)} |F_g(\log(n+c))|
\geq \sum_{m \notin A_\beta} \sum_{n \in C_\beta(m)} \frac{1}{n+c}. 
\end{gather*}
Furthermore, the series of the left-hand side is finite by \eqref{eq:04200421}, and therefore, 
\begin{gather}\label{eq:04201336}
\sum_{m \notin A_\beta} \sum_{n \in C_\beta(m)} \frac{1}{n+c}
< \infty. 
\end{gather}
If $m$ is sufficiently large, then the inner sum is evaluated as
\begin{gather}\label{eq:04201337}
\sum_{n \in C_\beta(m)} \frac{1}{n+c}
\geq \sum_{e^{m \beta}<n<e^{(m+1/2)\beta}-1} \frac{1}{n+1}
= \frac{\beta}{2}
+ o(1) 
\end{gather}
by the well-known formula $\sum_{n \leq x} 1/n= \log{x}+\gamma+o(1)$ as $x \to\infty$. 
Here, $\gamma$ is the Euler constant. 
By \eqref{eq:04201336} and \eqref{eq:04201337}, we find that the complement of $A_\beta$ must be a finite set. 
Denote by $a_m$ the $m$-th integer in the set $A_\beta$. 
Then we have $a_m/m \to 1$ as $m \to\infty$. 
By the definition of $A_\beta$, there exists $\lambda_m \in \mathbb{R}$ for any $m \geq1$ such that 
\begin{gather*}
|F_g(\lambda_m)|
< e^{-\lambda_m}
\quad\text{and}\quad
a_m \beta 
< \lambda_m
< \left(a_m+\frac{1}{2}\right) \beta. 
\end{gather*}
Hence $\lambda_m$ satisfies condition $(\mathrm{c})$ of Lemma \ref{lem:4.5}. 
Furthermore, condition $(\mathrm{b})$ holds with $\delta=\beta/2$ since
\begin{gather*}
|\lambda_m-\lambda_n|
> (a_m-a_n) \beta -\frac{\beta}{2}
\geq (m-n)\beta -\frac{\beta}{2}
\geq \frac{\beta}{2} (m-n) 
\end{gather*}
for any $m>n$. 
As a result, we deduce from Lemma \ref{lem:4.5} that 
\begin{gather*}
\limsup_{r \to\infty} \frac{\log|F_g(r)|}{r}
= \limsup_{m \to\infty} \frac{\log|F_g(\lambda_m)|}{|\lambda_m|}. 
\end{gather*}
Then we finally obtain \eqref{eq:04200509}, since $|F_g(\lambda_m)|<e^{-\lambda_m}$ holds for any $m \geq1$. 
On the other hand, the function $F_g(z)$ is represented as
\begin{gather*}
F_g(z)
= \int_{\mathbb{C}} e^{zs} \,d \mu_g(s), 
\end{gather*}
where $d \mu_g(s)= 1_U(-s) \overline{g(-s)} d \sigma dt$. 
Note that $\mu_g$ is a complex measure on $(\mathbb{C}, \mathcal{B}(\mathbb{C}))$ which is supported on $-\overline{U}= \{-s \mid s \in \overline{U} \}$. 
The assumption $\overline{U} \subset D$ implies that $-\overline{U}$ is contained in the half-plane $\sigma>-1$ . 
Hence Lemma \ref{lem:4.6} yields
\begin{gather}\label{eq:04201356}
\limsup_{r \to\infty} \frac{\log|F_g(r)|}{r}
> -1, 
\end{gather}
and we obtain contradiction by \eqref{eq:04200509} and \eqref{eq:04201356}. 
As a result, $\overline{B(c,M)}= A^2(U)$ holds for any $M \geq0$, that is, the set $B(c,M)$ is dense in $A^2(U)$ for any $M \geq0$.
\end{proof}

The second result ensures that we can approximate an element in $\Gamma(\alpha,N)$ by an element in $\Gamma(c,N)$ uniformly for $\alpha \in \mathcal{A}_\rho(c)$ if $\rho$ is sufficiently small. 
The condition $|\alpha-c| \leq \rho$ in the definition of $\mathcal{A}_\rho(c)$ is required here. 

\begin{proposition}\label{prop:4.8}
Let $0<c<1$. 
Then, for every $\epsilon>0$, there exists a positive real number $\rho$ such that
\begin{gather*}
\sup_{\alpha \in \mathcal{A}_\rho(c)}
\left\| \sum_{n=0}^{N} \frac{\gamma_n}{(n+\alpha)^s} 
-\sum_{n=0}^{N} \frac{\gamma_n}{(n+c)^s}\right\|
< \epsilon
\end{gather*}
for any $N \geq0$ if $|\gamma_n|=1$ for $0 \leq n \leq N$, where $\rho$ depends only on $c,\epsilon,U$. 
\end{proposition}

\begin{proof}
First, we have
\begin{align}\label{eq:04210148}
\left\| \sum_{n=0}^{N} \frac{\gamma_n}{(n+\alpha)^s} 
-\sum_{n=0}^{N} \frac{\gamma_n}{(n+c)^s}\right\|
&\leq \sum_{n=0}^{N} \left\|(n+\alpha)^{-s}-(n+c)^{-s}\right\| \\
&= \sum_{n=0}^{N} \left\|\int_{c}^{\alpha} (-s) (n+u)^{-s-1} \,du \right\| \nonumber 
\end{align}
since $|\gamma_n|=1$ for $0 \leq n \leq N$. 
Recall that $|\alpha-c| \leq \rho$ is valid for any $\alpha \in \mathcal{A}_\rho(c)$. 
Hence the inequality
\begin{gather*}
\left|\int_{c}^{\alpha} (-s) (n+u)^{-s-1} \,du \right|
\leq \rho |s| \sup_{u \in (c-\rho, c+\rho)}\left|(n+u)^{-s-1}\right|
\end{gather*}
holds for any $\alpha \in \mathcal{A}_\rho(c)$. 
Assume that $\rho$ satisfies $0<\rho \leq \min\{c, 1-c\}/2$. 
Then we obtain
\begin{gather*}
\left|\int_{c}^{\alpha} (-s) (n+u)^{-s-1} \,du \right|
\leq \rho |s| (n+c/2)^{-3/2}
\end{gather*}
for any $\alpha \in \mathcal{A}_\rho(c)$ and any $s \in U$. 
Inserting this inequality to \eqref{eq:04210148}, we deduce
\begin{align*}
\left\| \sum_{n=0}^{N} \frac{\gamma_n}{(n+\alpha)^s} 
-\sum_{n=0}^{N} \frac{\gamma_n}{(n+c)^s}\right\|
\leq \rho \sqrt{K} \sum_{n=0}^{N} (n+c/2)^{-3/2}
\leq \rho \sqrt{K} \zeta(3/2, c/2), 
\end{align*}
where $K= \iint_{U} |s|^2 \,d \sigma dt$ is a positive constant determined only by $U$. 
Hence, if we assume further that $\rho$ satisfies $0<\rho<\epsilon \{\sqrt{K} \zeta(3/2, c/2)\}^{-1}$, then we have
\begin{gather*}
\left\| \sum_{n=0}^{N} \frac{\gamma_n}{(n+\alpha)^s} 
-\sum_{n=0}^{N} \frac{\gamma_n}{(n+c)^s}\right\|
< \epsilon
\end{gather*}
for any $\alpha \in \mathcal{A}_\rho(c)$ and any $N \geq0$. 
This is the desired result. 
\end{proof}

\begin{proof}[Proof of Theorem \ref{thm:4.1}]
Let $0<c<1$ and $f \in A^2(U)$. 
Let $M$ be a positive integer chosen later. 
Then the function
\begin{gather*}
g(s)
= f(s)
- \sum_{n=0}^{M} \frac{1}{(n+c)^s}
\end{gather*}
belongs to the space $A^2(U)$. 
Thus Proposition \ref{prop:4.7} yields $g \in \overline{B(c,M)}$, that is, there exists an element $h \in B(c,M)$ with $\|g-h\|<\epsilon/3$ for every $\epsilon>0$. 
Hence, by the definition of $B(c,M)$, there exists an integer $N_0 \geq M+1$ such that 
\begin{gather*}
\left\|\left(f(s) 
- \sum_{n=0}^{M} \frac{1}{(n+c)^s}\right)
- \sum_{M<n \leq N_0} \frac{\beta_n}{(n+c)^s} \right\|
< \frac{\epsilon}{3}
\end{gather*}
with some $|\beta_n| \leq1$ for $M<n \leq N_0$. 
Note that the integer $N_0$ depends only on $c,\epsilon,f,U$, and $M$. 
Put $\beta_n=0$ for $N_0<n \leq N$ with any integer $N> N_0$. 
Then we obtain the inequality
\begin{gather}\label{eq:04210224}
\left\|f(s) 
- \sum_{n=0}^{M} \frac{1}{(n+c)^s}
- \sum_{M<n \leq N} \frac{\beta_n}{(n+c)^s} \right\|
< \frac{\epsilon}{3}. 
\end{gather}
We also deduce from Lemma \ref{lem:4.4} that 
\begin{gather*}
\left\|\sum_{M<n \leq N} \frac{\beta_n}{(n+c)^s}
- \sum_{M<n \leq N} \frac{\gamma_n}{(n+c)^s} \right\|^2
\leq 4 \sum_{M<n \leq N} \left\| (n+c)^{-s} \right\|^2
\end{gather*}
with some $|\gamma_n|=1$ for $M<n \leq N$. 
The sum of the right-hand side is estimated as 
\begin{gather*}
\sum_{M<n \leq N} \left\| (n+c)^{-s} \right\|^2
\leq \sum_{M<n \leq N} L (n+c)^{-2 \sigma_0} 
\leq \frac{L}{2\sigma_0-1} M^{1-2\sigma_0}, 
\end{gather*}
where $\sigma_0=\min\{\RE(s) \mid s \in \overline{U}\}$ and $L= \iint_{U} \,d \sigma dt$ are positive constants determined only by $U$. 
Here, we choose the integer $M=M(\epsilon,U)$ so that 
\begin{gather*}
\frac{L}{2\sigma_0-1} M^{1-2\sigma_0}
< \frac{\epsilon^2}{36}
\end{gather*}
is satisfied. 
Then we derive
\begin{gather}\label{eq:04210225}
\left\|\sum_{M<n \leq N} \frac{\beta_n}{(n+c)^s}
- \sum_{M<n \leq N} \frac{\gamma_n}{(n+c)^s} \right\|
\leq 2 \left(\frac{L}{2\sigma_0-1} M^{1-2\sigma_0}\right)^{1/2}
< \frac{\epsilon}{3}. 
\end{gather}
Lastly, by Proposition \ref{prop:4.8}, there exists a positive real number $\rho$ depending only on $c,\epsilon,U$ such that
\begin{gather}\label{eq:04210226}
\sup_{\alpha \in \mathcal{A}_\rho(c)}
\left\| \sum_{n=0}^{N} \frac{\gamma_n}{(n+\alpha)^s} 
-\sum_{n=0}^{N} \frac{\gamma_n}{(n+c)^s}\right\|
< \frac{\epsilon}{3}, 
\end{gather}
where we put $\gamma_n=1$ for $0 \leq n \leq M$. 
Combining \eqref{eq:04210224}, \eqref{eq:04210225}, and \eqref{eq:04210226}, we arrive at the inequality
\begin{gather*}
\left\| f(s) 
-\sum_{n=0}^{N} \frac{\gamma_n}{(n+\alpha)^s}\right\|
< \epsilon
\end{gather*}
for any $N> N_0$ and any $\alpha \in \mathcal{A}_\rho(c)$. 
In other words, we obtain that
\begin{gather*}
f 
\in \bigcup_{N_0\geq0} \bigcap_{N> N_0} \bigcap_{\alpha \in \mathcal{A}_\rho(c)} \Gamma(\alpha,N)^{(\epsilon)}. 
\end{gather*}
Therefore the desired result follows. 
\end{proof}

\subsection{Support of the random Dirichlet polynomial $\zeta_N(s,\mathbb{Y}_\alpha)$}\label{sec:4.3}
For any function $f \in H(D)$, the restriction of $f$ to $U$ is an element of $A^2(U)$. 
Then we regard the random Dirichlet polynomial $\zeta_N(s,\mathbb{Y}_\alpha)$ defined in Section \ref{sec:2} as a random element valued on $A^2(U)$. 
The following result is a simple consequence of Theorem \ref{thm:4.1}. 

\begin{corollary}\label{cor:4.9}
Let $0<c<1$ and $f \in A^2(U)$. 
Then, for every $\epsilon>0$, there exist a positive real number $\rho$ and an integer $N_0 \geq 0$ such that 
\begin{gather*}
\mathbf{P} \left(\|\zeta_N(s,\mathbb{Y}_\alpha)-f(s)\|<\epsilon\right)
> 0
\end{gather*}
for any $N>N_0$ and any $\alpha \in \mathcal{A}_\rho(c)$, where $\rho$ and $N_0$ depend only on $c,\epsilon,f,U$. 
\end{corollary}

\begin{proof}
Recall that the random variables $\mathbb{Y}_\alpha(n_1), \ldots, \mathbb{Y}_\alpha(n_k)$ are independent for any distinct integers $n_1,\ldots,n_k \geq0$. 
Indeed, we derive by Fubini's theorem that
\begin{align*}
&\mathbf{E} [\phi_1(\mathbb{Y}_\alpha(n_1)) \cdots \phi_k(\mathbb{Y}_\alpha(n_k))] \\
&= \int_{\Omega_2} \prod_{j=1}^{k} \phi_j(\mathbb{Y}_\alpha(n_j)) \,d \mathbf{P}_2 
= \prod_{j=1}^{k} \int_{S^1} \phi_j(\omega_{n_j}) \,d \mathbf{m} (\omega_{n_j}) 
= \prod_{j=1}^{k} \mathbf{E} [\phi_j(\mathbb{Y}_\alpha(n_j))]
\end{align*}
for any measurable functions $\phi_1, \ldots, \phi_k$ on $S^1$. 
Hence, the support of $\zeta_N(s,\mathbb{Y}_\alpha)$ is calculated as
\begin{gather*}
\supp \zeta_N(s,\mathbb{Y}_\alpha)
= \overline{\sum_{n=0}^{N} \supp \left(\frac{\mathbb{Y}_\alpha(n)}{(n+\alpha)^s}\right)}
\end{gather*}
by \cite[Proposition B.10.8]{Kowalski2021}, where $\sum_{n=0}^{N} S_n$ denotes the set of all points $x_0+\cdots+x_N$ with $x_n \in S_n$ for $0 \leq n \leq N$. 
We have 
\begin{gather*}
\supp \left(\frac{\mathbb{Y}_\alpha(n)}{(n+\alpha)^s}\right)
= \left\{ \frac{\gamma_n}{(n+\alpha)^s} ~\middle|~ |\gamma_n|=1 \right\}
\end{gather*}
for any $0 \leq n \leq N$. 
Therefore, by the definition of the set $\Gamma(\alpha,N)$, we derive that $\supp \zeta_N(s,\mathbb{Y}_\alpha)=\overline{\Gamma(\alpha,N)}$. 
We see that $\Gamma(\alpha,N)$ is closed in the space $A^2(U)$. 
Indeed, the continuous map $\psi: \prod_{n=0}^{N} S^1 \to A^2(U)$ defined as
\begin{gather*}
(\gamma_0, \ldots, \gamma_N) \mapsto \sum_{n=0}^{N} \frac{\gamma_n}{(n+\alpha)^s}
\end{gather*}
is a closed map since $\prod_{n=0}^{N} S^1$ is compact and $A^2(U)$ is Hausdorff.  
Hence 
\begin{gather}\label{eq:04212018}
\supp \zeta_N(s,\mathbb{Y}_\alpha)
= \Gamma(\alpha,N)
\end{gather}
follows. 
Let $0<c<1$ and $f \in A^2(U)$. 
By Theorem \ref{thm:4.1}, there exist a positive real number $\rho$ and an integer $N_0 \geq 0$ such that $f \in \Gamma(\alpha,N)^{(\epsilon/2)}$ for any $N>N_0$ and any $\alpha \in \mathcal{A}_\rho(c)$. 
Here, $\rho$ and $N_0$ depend at most on $c,\epsilon,f,U$. 
Then, by \eqref{eq:04212018}, there exists an element $g \in \supp \zeta_N(s,\mathbb{Y}_\alpha)$ such that $\|f-g\|<\epsilon/2$. 
Note that the condition $g \in \supp \zeta_N(s,\mathbb{Y}_\alpha)$ implies 
\begin{gather*}
\mathbf{P} \left(\|\zeta_N(s,\mathbb{Y}_\alpha)-g(s)\|<\frac{\epsilon}{2}\right)
> 0. 
\end{gather*}
Therefore, we obtain
\begin{gather*}
\mathbf{P} \left(\|\zeta_N(s,\mathbb{Y}_\alpha)-f(s)\|<\epsilon\right)
\geq \mathbf{P} \left(\|\zeta_N(s,\mathbb{Y}_\alpha)-g(s)\|<\frac{\epsilon}{2}\right)
> 0,  
\end{gather*}
and the proof is completed. 
\end{proof}

\section{Proof of the main result}\label{sec:5}
Let $K$ be a compact subset of the strip $D$ with connected complement. 
Then there exists a bounded domain $U$ such that $K \subset U$ and $\overline{U} \subset D$, whose boundary $\partial U$ is a Jordan curve. 
In the following, we fix such a domain $U$, and put 
\begin{gather*}
\sigma_0
= \min \{\RE(s) \mid s \in \overline{U}\}. 
\end{gather*}
Then we have $1/2<\sigma_0<1$. 
Let $\omega_0 \in \Omega$ be any sample. 
Denote by $\theta_n \in [0,2\pi)$ the argument of the value $\mathbb{Y}_\alpha(n)(\omega_0) \in S^1$ for $0 \leq n \leq N$. 
Define the event $\Omega_0$ as
\begin{align}\label{eq:04222013}
\Omega_0
&= \Omega_0(\delta; N,\alpha,\omega_0) \\
&= \left\{\omega \in \Omega ~\middle|~ 
\text{$\mathbb{X}_\alpha(n)(\omega) \in A(\theta_n-\pi \delta,\theta_n+\pi \delta)$ for $0 \leq n \leq N$}\right\} \nonumber
\end{align}
for $0<\delta<1$, where $A(s,t)$ is the arc of $S^1$ as in \eqref{eq:04212057}. 
For any $\mathbb{C}$-valued random variable $\mathcal{X}$ defined on the probability space $(\Omega, \mathcal{F}, \mathbf{P})$, we define
\begin{gather*}
\mathbf{E}_{\Omega_0}[\mathcal{X}]
= \int_{\Omega_0} \mathcal{X} \,d\mathbf{P}. 
\end{gather*}
Before proceeding to the proof of Theorem \ref{thm:1.4}, we study the conditional square mean value $\mathbf{E}_{\Omega_0}\left[\|\zeta(s,\mathbb{X}_\alpha)-\zeta_N(s,\mathbb{X}_\alpha)\|^2\right]$ for $\alpha \in \mathcal{A} \setminus \mathcal{E}$, where $\mathcal{E}$ is a certain finite subset chosen suitably.

\subsection{Results on the Beurling--Selberg functions}\label{sec:5.1}
Let $\mathbf{1}_{(s,t)}$ denote the indicator function of an open interval $(s,t) \subset \mathbb{R}$. 
The Beurling--Selberg functions present a nice approximation of $\mathbf{1}_{(s,t)}$. 
Here, we collect several results used later. 
See \cite{Vaaler1985} for the details of proofs. 
First, we define the functions $H(z)$ and $K(z)$ as
\begin{gather*}
H(z)
= \left(\frac{\sin \pi z}{\pi}\right)^2 
\left\{ \sum_{m=-\infty}^{\infty} \frac{\sgn(m)}{(z-m)^2} +\frac{2}{z} \right\} 
\quad\text{and}\quad
K(z)
= \left(\frac{\sin \pi z}{\pi z}\right)^2 
\end{gather*}
for $z \in \mathbb{C}$. 
They are entire functions of exponential type such that
\begin{gather*}
\limsup_{r \to\infty} \frac{\log|H(re^{i \theta})|}{r}
\leq 2\pi, 
\qquad
\limsup_{r \to\infty} \frac{\log|K(re^{i \theta})|}{r}
\leq 2\pi
\end{gather*}
uniformly in $\theta \in \mathbb{R}$. 
It is known that the inequalities
\begin{gather}\label{eq:042312137}
|\sgn(x)-H(x)|
\leq K(x)
\quad\text{and}\quad
|H(x)|
\leq 1 
\end{gather}
hold for any $x \in \mathbb{R}$. 
Let $s, t, \Delta \in \mathbb{R}$ with $s<t$ and $\Delta>1$. 
Then we define 
\begin{align*}
U_{s,t}(z, \Delta)
&= \frac{1}{2} \left\{ H \left(\frac{\Delta}{2\pi}(z-s)\right) + H \left(\frac{\Delta}{2\pi}(t-z)\right) \right\}, \\
K_{s,t}(z, \Delta)
&= \frac{1}{2} \left\{ K \left(\frac{\Delta}{2\pi}(z-s)\right) + K \left(\frac{\Delta}{2\pi}(t-z)\right) \right\} 
\end{align*}
for $z \in \mathbb{C}$. 
We deduce from \eqref{eq:042312137} the inequality
\begin{gather}\label{eq:04232220}
|\mathbf{1}_{(s,t)}(x) - U_{s,t}(x, \Delta)| 
\leq K_{s,t}(x, \Delta)
\end{gather}
for any $x \in \mathbb{R}$. 
Define the Fourier transforms $\widetilde{U}_{s,t}(\xi,\Delta)$ and $\widetilde{K}_{s,t}(\xi,\Delta)$ as 
\begin{gather*}
\widetilde{U}_{s,t}(\xi,\Delta)
= \int_{\mathbb{R}} U_{s,t}(x, \Delta) e^{-ix \xi} \,dx, 
\qquad
\widetilde{K}_{s,t}(\xi,\Delta)
= \int_{\mathbb{R}} K_{s,t}(x, \Delta) e^{-ix \xi} \,dx
\end{gather*}
for $\xi \in \mathbb{R}$. 
The Paley--Wiener theorem \cite[Theorem 19.3]{Rudin1987} yields that $\widetilde{U}_{s,t}(\xi, \Delta)=\widetilde{K}_{s,t}(\xi, \Delta)=0$ for any $|\xi|>\Delta$ since $U_{s,t}(z, \Delta)$ and $K_{s,t}(z, \Delta)$ are entire functions of exponential type such that
\begin{gather*}
\limsup_{r \to\infty} \frac{\log|U_{s,t}(r e^{i \theta}, \Delta)|}{r}
\leq \Delta, 
\qquad
\limsup_{r \to\infty} \frac{\log|K_{s,t}(r e^{i \theta}, \Delta)|}{r}
\leq \Delta
\end{gather*}
uniformly in $\theta \in \mathbb{R}$. 
Then, we define the functions $\mathscr{U}_{s,t}(z, \Delta)$ and $\mathscr{K}_{s,t}(z, \Delta)$ as
\begin{gather*}
\mathscr{U}_{s,t}(z, \Delta)
= \frac{1}{2\pi} \sum_{|m| \leq \Delta} \widetilde{U}_{s,t}(m, \Delta) z^m, 
\qquad
\mathscr{K}_{s,t}(z, \Delta)
= \frac{1}{2\pi} \sum_{|m| \leq \Delta} \widetilde{K}_{s,t}(m, \Delta) z^m 
\end{gather*}
for $z \in S^1$. 
We have the Fourier series representation
\begin{align}
\sum_{k \in \mathbb{Z}} U_{s,t}(\theta+2k \pi, \Delta)
&= \frac{1}{2\pi} \sum_{m \in \mathbb{Z}} \widetilde{U}_{s,t}(m, \Delta) e^{im \theta}
= \mathscr{U}_{s,t}(e^{i \theta}, \Delta), \label{eq:04232221}\\
\sum_{k \in \mathbb{Z}} K_{s,t}(\theta+2k \pi, \Delta)
&= \frac{1}{2\pi} \sum_{m \in \mathbb{Z}} \widetilde{K}_{s,t}(m, \Delta) e^{im \theta}
= \mathscr{K}_{s,t}(e^{i \theta}, \Delta) \label{eq:04232222}
\end{align}
for any $\theta \in \mathbb{R}$. 
Thus $\mathscr{U}_{s,t}(z, \Delta)$ and $\mathscr{K}_{s,t}(z, \Delta)$ are real valued functions. 

\begin{lemma}\label{lem:5.1}
Let $N \geq0$ be any integer. 
Let $s_n,t_n \in \mathbb{R}$ for $0 \leq n \leq N$ satisfying $0<t_n-s_n \leq 2\pi$. 
Let $\Delta>3$. 
Then we have
\begin{gather*}
\left| \prod_{n=0}^{N} \mathbf{1}_{A(s_n,t_n)}(z_n) 
- \prod_{n=0}^{N} \mathscr{U}_{s_n,t_n}(z_n, \Delta) \right| 
\ll (\log{\Delta})^{N+1} \sum_{n=0}^{N} \mathscr{K}_{s_n,t_n}(z_n, \Delta) 
\end{gather*}
for any $(z_0,\ldots,z_N) \in \prod_{n=0}^{N} S^1$, where $\mathbf{1}_{A(s,t)}$ denotes the indicator function of the arc $A(s,t)$ as in \eqref{eq:04212057}. 
Here, the implied constant is absolute. 
\end{lemma}

\begin{proof}
Let $0 \leq n \leq N$. 
Note that the formula
\begin{gather*}
\mathbf{1}_{A(s_n,t_n)}(e^{i \theta})
= \sum_{k \in \mathbb{Z}} \mathbf{1}_{(s_n,t_n)}(\theta+2k \pi) 
\end{gather*}
holds for any $\theta \in \mathbb{R}$. 
Hence we have
\begin{align}\label{eq:04212235}
&\left|\mathbf{1}_{A(s_n,t_n)}(e^{i \theta}) - \mathscr{U}_{s_n,t_n}(e^{i \theta}, \Delta) \right| \\
&\leq \sum_{k \in \mathbb{Z}} \left|\mathbf{1}_{A(s_n,t_n)}(\theta+2k \pi)-U_{s_n,t_n}(\theta+2k \pi, \Delta) \right| 
\leq \mathscr{K}_{s_n,t_n}(e^{i \theta}, \Delta) \nonumber
\end{align}
by \eqref{eq:04232220}, \eqref{eq:04232221}, and \eqref{eq:04232222}. 
Then, we prove the desired estimate by induction on $N$. 
Notice that \eqref{eq:04212235} derives the result for $N=0$. 
Furthermore, we obtain
\begin{align*}
&\left| \prod_{n=0}^{N+1} \mathbf{1}_{A(s_n,t_n)}(z_n) 
- \prod_{n=0}^{N+1} \mathscr{U}_{s_n,t_n}(z_n, \Delta) \right| \\
&\leq \prod_{n=0}^{N} \mathbf{1}_{A(s_n,t_n)}(z_n) 
\cdot \Big| \mathbf{1}_{A(s_{N+1},t_{N+1})}(z_{N+1}) - \mathscr{U}_{s_{N+1},t_{N+1}}(z_{N+1}, \Delta) \Big| \\
&\qquad
+ \left| \prod_{n=0}^{N} \mathbf{1}_{A(s_n,t_n)}(z_n) 
- \prod_{n=0}^{N} \mathscr{U}_{s_n,t_n}(z_n, \Delta) \right| 
\cdot \Big|\mathscr{U}_{s_{N+1},t_{N+1}}(z_{N+1}, \Delta)\Big| 
\end{align*}
for any $(z_0,\ldots,z_{N+1}) \in \prod_{n=0}^{N+1} S^1$. 
By the inductive assumption, it yields 
\begin{align}\label{eq:04212343}
&\left| \prod_{n=0}^{N+1} \mathbf{1}_{A(s_n,t_n)}(z_n) 
- \prod_{n=0}^{N+1} \mathscr{U}_{s_n,t_n}(z_n, \Delta) \right| \\
&\ll \mathscr{K}_{s_{N+1},t_{N+1}}(z_{N+1}, \Delta) \nonumber\\
&\qquad
+ (\log{\Delta})^{N+1} \sum_{n=0}^{N} \mathscr{K}_{s_n,t_n}(z_n, \Delta) 
\cdot \Big|\mathscr{U}_{s_{N+1},t_{N+1}}(z_{N+1}, \Delta)\Big|. \nonumber
\end{align}
By the definition of the function $\mathscr{U}_{s,t}(z, \Delta)$, we see that
\begin{gather}\label{eq:04212341}
\Big|\mathscr{U}_{s,t}(z, \Delta)\Big|
\leq \frac{1}{2\pi} \sum_{|m| \leq \Delta} \left|\widetilde{U}_{s,t}(m, \Delta)\right|. 
\end{gather}
for any $s,t \in \mathbb{R}$ with $0<t-s \leq 2\pi$ and any $z \in S^1$. 
For $m=0$, we have
\begin{gather*}
\widetilde{U}_{s,t}(0, \Delta)
= \int_{\mathbb{R}} U_{s,t}(x, \Delta) \,dx
\leq \int_{\mathbb{R}} \mathbf{1}_{(s,t)}(x) \,dx 
+ \int_{\mathbb{R}} K_{s,t}(x, \Delta) \,dx
\end{gather*}
by \eqref{eq:04232220}. 
Here, the first integral of the right-hand side is bounded by $2\pi$, and the second integral is 
\begin{gather}\label{eq:04220117}
\int_{\mathbb{R}} K_{s,t}(x, \Delta) \,dx
= \int_{\mathbb{R}} K \left(\frac{\Delta x}{2\pi} \right) \,dx
= \frac{2\pi}{\Delta} \int_{\mathbb{R}} \left(\frac{\sin \pi x}{\pi x}\right)^2 \,dx
\ll \frac{1}{\Delta}
\end{gather}
by the definition of $K_{s,t}(x, \Delta)$. 
Hence we have $\widetilde{U}_{s,t}(0, \Delta) \ll 1$ with an absolute implied constant. 
For $m \neq 0$, we obtain
\begin{gather*}
\widetilde{U}_{s,t}(m, \Delta)
= -\frac{1}{im} \int_{\mathbb{R}} \left(\frac{\partial}{\partial x} U_{s,t}(x, \Delta)\right) e^{-imx} \,dx
\end{gather*}
by integrating by parts. 
Here, we apply the estimate $H'(x) \ll (1+|x|)^{-3}$ proved in \cite[Theorem 6]{Vaaler1985}. 
By the definition of $U_{s,t}(x, \Delta)$, it yields
\begin{gather*}
\frac{\partial}{\partial x} U_{s,t}(x, \Delta)
\ll \frac{\Delta}{(1+\Delta|x-s|)^3} + \frac{\Delta}{(1+\Delta|x-t|)^3} 
\end{gather*}
uniformly for $x \in \mathbb{R}$. 
Hence we have $\widetilde{U}_{s,t}(m, \Delta) \ll 1/|m|$ with an absolute implied constant. 
By \eqref{eq:04212341}, we obtain
\begin{gather*}
\Big|\mathscr{U}_{s,t}(z, \Delta)\Big|
\ll 1+\sum_{0<m \leq \Delta} \frac{1}{m}
\ll \log{\Delta} 
\end{gather*}
for any $s,t \in \mathbb{R}$ with $0<t-s \leq 2\pi$ and any $z \in S^1$. 
Inserting this to \eqref{eq:04212343}, we derive
\begin{align*}
&\left| \prod_{n=0}^{N+1} \mathbf{1}_{A(s_n,t_n)}(z_n) 
- \prod_{n=0}^{N+1} \mathscr{U}_{s_n,t_n}(z_n, \Delta) \right| \\
&\ll \mathscr{K}_{s_{N+1},t_{N+1}}(z_{N+1}, \Delta)
+ (\log{\Delta})^{N+2} \sum_{n=0}^{N} \mathscr{K}_{s_n,t_n}(z_n, \Delta) \\
&\ll (\log{\Delta})^{N+2} \sum_{n=0}^{N+1} \mathscr{K}_{s_n,t_n}(z_n, \Delta).  
\end{align*}
This is the result for $N+1$, and hence the proof is completed. 
\end{proof}

\subsection{Estimate of a conditional square mean value}\label{sec:5.2}
By the definition of the event $\Omega_0$ as in \eqref{eq:04222013}, the indicator function $\mathbf{1}_{\Omega_0}$ is represented as
\begin{gather*}
\mathbf{1}_{\Omega_0}
= \prod_{n=0}^{N} \mathbf{1}_{A(s_n,t_n)}(\mathbb{X}_\alpha(n)), 
\end{gather*}
where we put $s_n=\theta_n-\pi \delta$ and $t_n=\theta_n+\pi \delta$. 
Hence Lemma \ref{lem:5.1} yields 
\begin{gather}\label{eq:04220048}
\mathbf{1}_{\Omega_0}
= \prod_{n=0}^{N} \mathscr{U}_{s_n,t_n}(\mathbb{X}_\alpha(n), \Delta) 
+ O\left( (\log{\Delta})^{N+1} \sum_{n=0}^{N} \mathscr{K}_{s_n,t_n}(\mathbb{X}_\alpha(n), \Delta) \right), 
\end{gather}
where the implied constant is absolute. 
Using this asymptotic formula, we prove the following result which plays an important role in the remaining part of the proof of Theorem \ref{thm:1.4}. 

\begin{proposition}\label{prop:5.2}
Let $0<c<1$ and $0<\delta<1/2$. 
Let $\omega_0 \in \Omega$ be any sample. 
For any $L>N \geq1$ and any $\Delta>\Delta_0$ with some absolute constant $\Delta_0$, there exists a finite subset $\mathcal{E} \subset \mathcal{A}$ such that 
\begin{gather}
\mathbf{P}(\Omega_0)
= \delta^{N+1}+O\left(\frac{N(\log{\Delta})^{N+1}}{\Delta}\right), \label{eq:04220227}\\
\mathbf{E}_{\Omega_0}\left[\|\zeta(s,\mathbb{X}_\alpha)-\zeta_N(s,\mathbb{X}_\alpha)\|^2\right] 
\ll \mathbf{P}(\Omega_0) N^{1-2\sigma_0}
+ L^{1-2\sigma_0}
+ \frac{NL (\log{\Delta})^{N+1}}{\Delta} \label{eq:04220228}
\end{gather}
for any $\alpha \in \mathcal{A} \setminus \mathcal{E}$, where $\mathcal{E}$ depends only on $N$, $L$, and $\Delta$. 
The implied constants depend only on the compact subset $K$. 
\end{proposition}

\begin{proof}
First, we evaluate the probability $\mathbf{P}(\Omega_0)$. 
Using \eqref{eq:04220048}, we obtain the formula
\begin{align}\label{eq:04240203}
\mathbf{P}(\Omega_0) 
&= \mathbf{E} \left[\prod_{n=0}^{N} \mathscr{U}_{s_n,t_n}(\mathbb{X}_\alpha(n), \Delta) \right] \\
&\qquad
+ O\left( (\log{\Delta})^{N+1} \sum_{n=0}^{N} \mathbf{E} \left[\mathscr{K}_{s_n,t_n}(\mathbb{X}_\alpha(n), \Delta)\right] \right). \nonumber
\end{align}
By the definition of the function $\mathscr{U}_{s,t}(z, \Delta)$, the first term of the right-hand side is calculated as
\begin{align}\label{eq:04220115}
&\mathbf{E} \left[\prod_{n=0}^{N} \mathscr{U}_{s_n,t_n}(\mathbb{X}_\alpha(n), \Delta) \right] \\
&= \left(\frac{1}{2\pi}\right)^{N+1} 
\sum_{|m_0| \leq \Delta} \cdots \sum_{|m_N| \leq \Delta}
\prod_{n=0}^{N} \widetilde{U}_{s_n,t_n}(m_n, \Delta) \,
\mathbf{E}\left[\prod_{n=0}^{N} \mathbb{X}_\alpha(n)^{m_n}\right]. \nonumber
\end{align}
Define the set $\mathcal{E}_1$ as
\begin{gather*}
\mathcal{E}_1
= \bigcup_{\substack{(m_0,\ldots,m_N) \in \mathbb{Z}^{N+1} \setminus\{\mathbf{0}\} \\ \forall n,~ |m_n| \leq \Delta}}
\left\{ \alpha \in \mathcal{A} ~\middle|~ \prod_{n=0}^{N}(n+\alpha)^{m_n}=1 \right\}. 
\end{gather*}
Then $\mathcal{E}_1$ is a finite subset of $\mathcal{A}$ which is determined only by $N$ and $\Delta$. 
Suppose $\alpha \in \mathcal{A} \setminus \mathcal{E}_1$. 
By the definition of $\mathcal{E}_1$, we see that $\prod_{n=0}^{N}(n+\alpha)^{m_n}\neq 1$ is satisfied unless $m_0=\cdots=m_N=0$ for any $|m_n| \leq \Delta$. 
Hence Lemma \ref{lem:2.1} yields
\begin{gather*}
\mathbf{E}\left[\prod_{n=0}^{N} \mathbb{X}_\alpha(n)^{m_n}\right]
=
\begin{cases}
1
& \text{if $m_0=\cdots=m_N=0$}
\\
0 
& \text{otherwise}
\end{cases}
\end{gather*}
for any $|m_n| \leq \Delta$. 
Inserting this to \eqref{eq:04220115}, we derive
\begin{gather*}
\mathbf{E} \left[\prod_{n=0}^{N} \mathscr{U}_{s_n,t_n}(\mathbb{X}_\alpha(n), \Delta) \right]
= \left(\frac{1}{2\pi}\right)^{N+1} 
\prod_{n=0}^{N} \widetilde{U}_{s_n,t_n}(0, \Delta). 
\end{gather*}
Then we evaluate $\widetilde{U}_{s_n,t_n}(0, \Delta)$ for $0 \leq n \leq N$. 
Note that $t_n-s_n=2\pi \delta$ is satisfied by the setting of $s_n$ and $t_n$. 
Applying \eqref{eq:04232220}, we have 
\begin{align*}
\widetilde{U}_{s_n,t_n}(0, \Delta)
&= \int_{\mathbb{R}} U_{s_n,t_n}(x, \Delta) \,dx \\
&= \int_{\mathbb{R}} \mathbf{1}_{(s_n,t_n)}(x) \,dx 
+ O\left(\int_{\mathbb{R}} K_{s_n,t_n}(x, \Delta) \,dx\right) \\
&= 2\pi \delta
+ O\left(\frac{1}{\Delta}\right), 
\end{align*}
where the last line follows from \eqref{eq:04220117}. 
As a result, we obtain
\begin{align}\label{eq:04220225}
\mathbf{E} \left[\prod_{n=0}^{N} \mathscr{U}_{s_n,t_n}(\mathbb{X}_\alpha(n), \Delta) \right]
&= \left(\frac{1}{2\pi}\right)^{N+1} 
\prod_{n=0}^{N} \left(2\pi \delta + O\left(\frac{1}{\Delta}\right)\right) \\
&= \delta^{N+1}
+ O\left(\frac{N}{\Delta}\right) \nonumber
\end{align}
for any $\Delta>\Delta_0$ with a large absolute constant $\Delta_0$. 
On the other hand, we have
\begin{gather*}
\mathbf{E}[\mathscr{K}_{s_n,t_n}(\mathbb{X}_\alpha(n), \Delta)]
= \frac{1}{2\pi} \sum_{|m| \leq \Delta} \widetilde{K}_{s_n,t_n}(m, \Delta) 
\mathbf{E}\left[\mathbb{X}_\alpha(n)^m \right]
= \frac{1}{2\pi} \widetilde{K}_{s_n,t_n}(0, \Delta) 
\end{gather*}
since Lemma \ref{lem:2.1} yields $\mathbf{E}\left[\mathbb{X}_\alpha(n)^m \right]=0$ unless $m=0$. 
Furthermore, 
\begin{gather*}
\widetilde{K}_{s_n,t_n}(0, \Delta) 
= \int_{\mathbb{R}} K_{s_n,t_n}(x, \Delta) \,dx
\ll \frac{1}{\Delta} 
\end{gather*}
by \eqref{eq:04220117}. 
Thus $\mathbf{E}[\mathscr{K}_{s_n,t_n}(\mathbb{X}_\alpha(n), \Delta)] \ll 1/\Delta$ follows. 
Then we arrive at
\begin{gather}\label{eq:04220226}
(\log{\Delta})^{N+1} \sum_{n=0}^{N} \mathbf{E} \left[\mathscr{K}_{s_n,t_n}(\mathbb{X}_\alpha(n), \Delta)\right]
\ll \frac{N(\log{\Delta})^{N+1}}{\Delta}. 
\end{gather}
Combining \eqref{eq:04240203}, \eqref{eq:04220225}, and \eqref{eq:04220226}, we see that \eqref{eq:04220227} holds for any $\alpha \in \mathcal{A} \setminus \mathcal{E}_1$. 

Next, we consider the conditional square mean value of \eqref{eq:04220228}. 
We divide it into the following two expected values: 
\begin{align}\label{eq:04240222}
&\mathbf{E}_{\Omega_0}\left[\|\zeta(s,\mathbb{X}_\alpha)-\zeta_N(s,\mathbb{X}_\alpha)\|^2\right] \\
&\leq \mathbf{E}_{\Omega_0}\left[\big(\|\zeta(s,\mathbb{X}_\alpha)-\zeta_L(s,\mathbb{X}_\alpha)\|
+ \|\zeta_L(s,\mathbb{X}_\alpha)-\zeta_N(s,\mathbb{X}_\alpha)\|\big)^2\right] \nonumber\\
&\leq 2\mathbf{E}_{\Omega_0}\left[\|\zeta(s,\mathbb{X}_\alpha)-\zeta_L(s,\mathbb{X}_\alpha)\|^2\right]
+ 2\mathbf{E}_{\Omega_0}\left[\|\zeta_L(s,\mathbb{X}_\alpha)-\zeta_N(s,\mathbb{X}_\alpha)\|^2\right]. \nonumber
\end{align}
The first expected value is evaluated as
\begin{align*}
\mathbf{E}_{\Omega_0}\left[\|\zeta(s,\mathbb{X}_\alpha)-\zeta_L(s,\mathbb{X}_\alpha)\|^2\right] 
&= \mathbf{E}\left[\mathbf{1}_{\Omega_0} \cdot \|\zeta(s,\mathbb{X}_\alpha)-\zeta_L(s,\mathbb{X}_\alpha)\|^2\right] \\
&\leq \mathbf{E}\left[\|\zeta(s,\mathbb{X}_\alpha)-\zeta_L(s,\mathbb{X}_\alpha)\|^2\right]. 
\end{align*}
Furthermore, we have
\begin{align*}
\mathbf{E} \left[\|\zeta(s,\mathbb{X}_\alpha)-\zeta_L(s,\mathbb{X}_\alpha)\|^2\right]
&= \sum_{m,n>L} \mathbf{E}[\mathbb{X}_\alpha(m) \overline{\mathbb{X}_\alpha(n)}]
\left\langle (m+\alpha)^{-s}, (n+\alpha)^{-s} \right\rangle \\
&= \sum_{n>L} \|(n+\alpha)^{-s}\|^2 
\end{align*}
since $\mathbf{E}[\mathbb{X}_\alpha(m) \overline{\mathbb{X}_\alpha(n)}]=0$ for $m \neq n$ by Lemma \ref{lem:2.1}. 
Then we see that
\begin{gather}\label{eq:04220320}
\sum_{n>L} \|(n+\alpha)^{-s}\|^2
\leq \sum_{n>L} M (n+\alpha)^{-2\sigma_0}
\leq \frac{M}{2\sigma_0-1} L^{1-2\sigma_0}, 
\end{gather}
where $M= \iint_{U} \,d \sigma dt$ is a positive constant determined only by $K$. 
From the above, we deduce
\begin{gather}\label{eq:04220358}
\mathbf{E}_{\Omega_0}\left[\|\zeta(s,\mathbb{X}_\alpha)-\zeta_L(s,\mathbb{X}_\alpha)\|^2\right] 
\ll L^{1-2\sigma_0}
\end{gather}
with an implied constant depending only on $K$. 
On the other hand, we obtain
\begin{align}\label{eq:04240235}
&\mathbf{E}_{\Omega_0} \left[\|\zeta_L(s,\mathbb{X}_\alpha)-\zeta_N(s,\mathbb{X}_\alpha)\|^2\right] \\
&= \sum_{N<n_1,n_2 \leq L} 
\mathbf{E}_{\Omega_0}[\mathbb{X}_\alpha(n_1) \overline{\mathbb{X}_\alpha(n_2)}]
\left\langle (n_1+\alpha)^{-s}, (n_2+\alpha)^{-s} \right\rangle \nonumber\\
&= \mathbf{P}(\Omega_0) \sum_{N<n \leq L} \|(n+\alpha)^{-s}\|^2 \nonumber\\
&\qquad
+ \sum_{\substack{N<n_1,n_2 \leq L \\ n_1 \neq n_2}}
\mathbf{E}_{\Omega_0}[\mathbb{X}_\alpha(n_1) \overline{\mathbb{X}_\alpha(n_2)}]
\left\langle (n_1+\alpha)^{-s}, (n_2+\alpha)^{-s} \right\rangle. \nonumber
\end{align}
In a similar way that we obtain \eqref{eq:04220320}, the first term is estimated as
\begin{gather}\label{eq:04220356}
\mathbf{P}(\Omega_0) \sum_{N<n \leq L} \|(n+\alpha)^{-s}\|^2
\ll \mathbf{P}(\Omega_0) N^{1-2\sigma_0}, 
\end{gather}
where the implied constant depends only on $K$. 
Let $N<n_1,n_2 \leq L$ with $n_1 \neq n_2$. 
By \eqref{eq:04220048}, we obtain the formula
\begin{align}\label{eq:04220344}
\mathbf{E}_{\Omega_0}[\mathbb{X}_\alpha(n_1) \overline{\mathbb{X}_\alpha(n_2)}] 
&= \mathbf{E}\left[\mathbb{X}_\alpha(n_1) \overline{\mathbb{X}_\alpha(n_2)}
\prod_{n=0}^{N} \mathscr{U}_{s_n,t_n}(\mathbb{X}_\alpha(n), \Delta) \right] \\
&\qquad
+ O\left((\log{\Delta})^{N+1} \sum_{n=0}^{N} \mathbf{E} \left[\mathscr{K}_{s_n,t_n}(\mathbb{X}_\alpha(n), \Delta)\right]\right). \nonumber
\end{align}
Here, the first term on the right-hand side is calculated as
\begin{align}\label{eq:04240214}
&\mathbf{E}\left[\mathbb{X}_\alpha(n_1) \overline{\mathbb{X}_\alpha(n_2)}
\prod_{n=0}^{N} \mathscr{U}_{s_n,t_n}(\mathbb{X}_\alpha(n), \Delta) \right] \\
&= \left(\frac{1}{2\pi}\right)^{N+1} 
\sum_{|m_0| \leq \Delta} \cdots \sum_{|m_N| \leq \Delta} \nonumber\\
&\hspace{30mm} 
\times \prod_{n=0}^{N} \widetilde{U}_{s_n,t_n}(m_n, \Delta)
\mathbf{E}\left[\mathbb{X}_\alpha(n_1) \overline{\mathbb{X}_\alpha(n_2)} \prod_{n=0}^{N} \mathbb{X}_\alpha(n)^{m_n}\right]. \nonumber
\end{align}
Define the set $\mathcal{E}_2$ as 
\begin{gather*}
\mathcal{E}_2
= \bigcup_{\substack{(m_0,\ldots,m_N) \in\, \mathbb{Z}^{N+1} \\ \forall n,~ |m_n| \leq \Delta}}
\bigcup_{\substack{N<n_1,n_2<L \\ n_1 \neq n_2}}
\left\{ \alpha \in \mathcal{A} ~\middle|~ \prod_{n=0}^{N}(n+\alpha)^{m_n}=\frac{n_2+\alpha}{n_1+\alpha} \right\}. 
\end{gather*}
Then $\mathcal{E}_2$ is a finite subset of $\mathcal{A}$ which is determined only by $N$, $L$, and $\Delta$. 
Suppose $\alpha \in \mathcal{A} \setminus \mathcal{E}_2$. 
By the definition of $\mathcal{E}_2$, we see that 
\begin{gather*}
(n_1+\alpha) (n_2+\alpha)^{-1} \prod_{n=0}^{N}(n+\alpha)^{m_n}
= 1
\end{gather*}
never holds for any $|m_n| \leq \Delta$. 
Therefore, Lemma \ref{lem:2.1} yields that all expected values appearing in \eqref{eq:04240214} are equal to $0$. 
Thus we find that
\begin{gather}\label{eq:04220343}
\mathbf{E}\left[\mathbb{X}_\alpha(n_1) \overline{\mathbb{X}_\alpha(n_2)}
\prod_{n=0}^{N} \mathscr{U}_{s_n,t_n}(\mathbb{X}_\alpha(n), \Delta) \right]
= 0. 
\end{gather}
By \eqref{eq:04220343} and \eqref{eq:04220226}, it is deduced from \eqref{eq:04220344} that
\begin{gather*}
\mathbf{E}_{\Omega_0}[\mathbb{X}_\alpha(n_1) \overline{\mathbb{X}_\alpha(n_2)}] 
\ll \frac{N(\log{\Delta})^{N+1}}{\Delta} 
\end{gather*}
for any $N<n_1,n_2 \leq L$ with $n_1 \neq n_2$, where the implied constant is absolute. 
Hence we obtain
\begin{align*}
&\sum_{\substack{N<n_1,n_2 \leq L \\ n_1 \neq n_2}}
\mathbf{E}_{\Omega_0}[\mathbb{X}_\alpha(n_1) \overline{\mathbb{X}_\alpha(n_2)}]
\left\langle (n_1+\alpha)^{-s}, (n_2+\alpha)^{-s} \right\rangle \\
&\ll \frac{N(\log{\Delta})^{N+1}}{\Delta}
\sum_{\substack{N<n_1,n_2 \leq L \\ n_1 \neq n_2}} 
\left|\left\langle (n_1+\alpha)^{-s}, (n_2+\alpha)^{-s} \right\rangle \right|. 
\end{align*}
Furthermore, we have
\begin{align*}
\sum_{\substack{N<n_1,n_2 \leq L \\ n_1 \neq n_2}} 
\left|\left\langle (n_1+\alpha)^{-s}, (n_2+\alpha)^{-s} \right\rangle \right|
&\leq \sum_{N<n_1,n_2 \leq L} 
\left\|(n_1+\alpha)^{-s}\right\| \left\|(n_2+\alpha)^{-s} \right\| \\
&\leq \left(\sum_{N<n \leq L} \|(n+\alpha)^{-s}\|\right)^2
\end{align*}
by the Cauchy--Schwarz inequality. 
Here, the last sum is evaluated as
\begin{gather*}
\sum_{N<n \leq L} \|(n+\alpha)^{-s}\|
\leq \sum_{N<n \leq L} \sqrt{M} (n+\alpha)^{-1/2} 
\leq 2 \sqrt{M}\sqrt{L}, 
\end{gather*}
where we put $M= \iint_{U} \,d \sigma dt$ as before. 
Therefore we obtain
\begin{gather}\label{eq:04220357}
\sum_{\substack{N<n_1,n_2 \leq L \\ n_1 \neq n_2}}
\mathbf{E}_{\Omega_0}[\mathbb{X}_\alpha(n_1) \overline{\mathbb{X}_\alpha(n_2)}]
\left\langle (n_1+\alpha)^{-s}, (n_2+\alpha)^{-s} \right\rangle
\ll \frac{NL(\log{\Delta})^{N+1}}{\Delta},
\end{gather}
where the implied constant depends only on $K$. 
By \eqref{eq:04220356} and \eqref{eq:04220357}, it is deduced from \eqref{eq:04240235} that 
\begin{gather}\label{eq:04220359}
\mathbf{E}_{\Omega_0} \left[\|\zeta_L(s,\mathbb{X}_\alpha)-\zeta_N(s,\mathbb{X}_\alpha)\|^2\right]
\ll \mathbf{P}(\Omega_0) N^{1-2\sigma_0}
+ \frac{NL(\log{\Delta})^{N+1}}{\Delta}. 
\end{gather}
Combining \eqref{eq:04240222}, \eqref{eq:04220358}, and \eqref{eq:04220359}, we derive \eqref{eq:04220228} for any $\alpha \in \mathcal{A} \setminus \mathcal{E}_2$. 
Then we obtain the desired result by letting $\mathcal{E}=\mathcal{E}_1 \cup \mathcal{E}_2$. 
\end{proof}

\begin{remark}\label{rem:5.3}
Let $\delta, N,\alpha,\omega_0$ be in Proposition \ref{prop:5.2}. 
Define the event $\Omega'_0 \subset \Omega$ as 
\begin{align*}
\Omega'_0
&= \Omega'_0(\delta; N,\alpha,\omega_0) \\
&= \left\{\omega \in \Omega ~\middle|~ 
\text{$\mathbb{Y}_\alpha(n)(\omega) \in A(\theta_n-\pi \delta,\theta_n+\pi \delta)$ for any $0 \leq n \leq N$}\right\} \nonumber
\end{align*}
as an analogue of \eqref{eq:04222013}. 
The probability $\mathbf{P}(\Omega'_0)$ is calculated as
\begin{gather*}
\mathbf{P}(\Omega'_0)
= \prod_{n=0}^{N} \mathbf{P} \left(\mathbb{Y}_\alpha(n) \in A(\theta_n-\pi \delta,\theta_n+\pi \delta)\right)
= \delta^{N+1}
\end{gather*}
for any $\alpha \in \mathcal{A}$ since $\mathbb{Y}_\alpha(0), \ldots, \mathbb{Y}_\alpha(N)$ are independent. 
Furthermore, we have
\begin{align*}
\mathbf{E}_{\Omega'_0}\left[\|\zeta(s,\mathbb{Y}_\alpha)-\zeta_N(s,\mathbb{Y}_\alpha)\|^2\right] 
&= \mathbf{E}[\mathbf{1}_{\Omega'_0}] \cdot \mathbf{E}\left[\|\zeta(s,\mathbb{Y}_\alpha)-\zeta_N(s,\mathbb{Y}_\alpha)\|^2\right] \\
&\ll \mathbf{P}(\Omega'_0) N^{1-2\sigma_0}
\end{align*}
for any $\alpha \in \mathcal{A}$. 
Proposition \ref{prop:5.2} means that similar results are valid for all but finitely many $\alpha \in \mathcal{A}$ if we replace the above $\mathbb{Y}_\alpha(n)$ with $\mathbb{X}_\alpha(n)$. 
\end{remark}

\subsection{Proof of Theorem \ref{thm:1.4}}\label{sec:5.3}
In this section, we finally complete the proof of the main result. 
The last lemma required for the proof is the Mergelyan theorem, which is a complex analogue of the Weierstrass approximation theorem. 

\begin{lemma}[Mergelyan theorem]\label{lem:5.4}
Let $K$ be a compact subset of $\mathbb{C}$ with connected complement. 
Let $f$ be a continuous function on $K$ which is analytic in the interior of $K$. 
Then, for every $\epsilon>0$, there exists a polynomial $p(s)$ such that
\begin{gather*}
\sup_{s \in K} \left|f(s)-p(s)\right|
< \epsilon. 
\end{gather*}
\end{lemma}

\begin{proof}
See \cite[Theorem 20.5]{Rudin1987} for a proof. 
\end{proof}

\begin{proof}[Proof of Theorem \ref{thm:1.4}]
Let $K,f,\epsilon$ be as in the statement of Theorem \ref{thm:1.4}. 
Then we deduce from Lemma \ref{lem:5.4} that
\begin{gather*}
\sup_{s \in K} \left|f(s)-p(s)\right|
< \frac{\epsilon}{2}
\end{gather*}
with some polynomial $p(s)$. 
Remark that $p(s)$ is obviously an element of $H(D)$. 
For any $\alpha \in \mathcal{A}$, we obtain
\begin{align}\label{eq:04221641}
&\frac{1}{T}\meas\left\{ \tau \in [0,T] ~\middle|~ \sup_{s \in K} |\zeta(s+i \tau, \alpha)-f(s)|<\epsilon \right\} \\
&\geq \frac{1}{T}\meas\left\{ \tau \in [0,T] ~\middle|~ \sup_{s \in K} |\zeta(s+i \tau, \alpha)-p(s)|<\frac{\epsilon}{2} \right\} \nonumber\\
&= P_{\alpha,T}(A), \nonumber
\end{align}
where the subset $A \subset H(D)$ is defined as 
\begin{gather*}
A
= \left\{g \in H(D) ~\middle|~ \sup_{s \in K}|g(s)-p(s)|<\frac{\epsilon}{2} \right\}. 
\end{gather*}
Here, we check that it is an open set of $H(D)$. 
Let $\{g_n\}$ be any sequence of functions in $H(D) \setminus A$ such that $g_n$ converges to $g \in H(D)$ as $n \to\infty$. 
Then $g_n(s)$ converges uniformly on $K$ by the definition of the topology of $H(D)$. 
Thus we obtain
\begin{gather*}
\sup_{s \in K}|g(s)-p(s)|
= \lim_{n \to\infty} \sup_{s \in K}|g_n(s)-p(s)|
\geq \frac{\epsilon}{2}, 
\end{gather*}
which implies $g \notin A$. 
Hence $H(D) \setminus A$ is closed, and equivalently, $A$ is open. 
Then it is deduced from Theorem \ref{thm:3.1} and Lemma \ref{lem:3.7} that
\begin{gather}\label{eq:04221642}
\liminf_{T \to\infty} P_{\alpha,T}(A)
\geq Q_\alpha(A)
= \mathbf{P}\left(\sup_{s \in K} |\zeta(s, \mathbb{X}_\alpha)-p(s)|<\frac{\epsilon}{2}\right). 
\end{gather}
Let $s_0 \in K$, and put $R=\min\{|z-s| \mid z \in \partial U, s \in K\}$. 
Here, $U$ is a fixed domain such that $K \subset U$ and $\overline{U} \subset D$, whose boundary $\partial U$ is a Jordan curve. 
Then Cauchy's integral formula yields that
\begin{align*}
\left|\zeta(s_0, \mathbb{X}_\alpha)(\omega)-p(s_0)\right|
&= \left|\frac{1}{2\pi i} \oint_{|z-s_0|=r} \frac{\zeta(z, \mathbb{X}_\alpha)(\omega)-p(z)}{z-s_0} \,dz\right| \\
&\leq \frac{1}{2\pi} \int_{0}^{2\pi} \left|\zeta(r e^{i \theta}, \mathbb{X}_\alpha)(\omega)-p(r e^{i \theta})\right| \,d \theta 
\end{align*}
for $\omega \in \Omega_\mathrm{u}$ and $0<r<R$, where $\Omega_\mathrm{u}$ is the same as in the proof of Proposition \ref{prop:2.4}. 
Therefore, we obtain
\begin{align*}
\left|\zeta(s_0, \mathbb{X}_\alpha)(\omega)-p(s_0)\right| 
&\leq \frac{1}{\pi R^2} \int_{0}^{R} \left(\int_{0}^{2\pi} \left|\zeta(r e^{i \theta}, \mathbb{X}_\alpha)(\omega)-p(r e^{i \theta})\right| \,d \theta\right) r \,dr \\
&\leq \frac{1}{\pi R^2} \iint_{U} \left|\zeta(s, \mathbb{X}_\alpha)(\omega)-p(s)\right| \,d \sigma dt \\
&\leq \frac{\sqrt{M}}{\pi R^2} \|\zeta(s, \mathbb{X}_\alpha)(\omega)-p(s)\|
\end{align*}
by the Cauchy--Schwarz inequality, where $M= \iint_{U} \,d \sigma dt$. 
Hence we have 
\begin{gather*}
\sup_{s \in K} |\zeta(s, \mathbb{X}_\alpha)(\omega)-p(s)|
\leq C_K \|\zeta(s, \mathbb{X}_\alpha)(\omega)-p(s)\|
\end{gather*}
with a positive constant $C_K$ depending only on $K$. 
It yields the inequality
\begin{gather}\label{eq:04221643}
\mathbf{P}\left(\sup_{s \in K} |\zeta(s, \mathbb{X}_\alpha)-p(s)|<\frac{\epsilon}{2}\right)
\geq \mathbf{P}\left(\|\zeta(s, \mathbb{X}_\alpha)-p(s)\|<\frac{\epsilon}{2C_K}\right). 
\end{gather}
Combining \eqref{eq:04221641}, \eqref{eq:04221642}, and \eqref{eq:04221643}, we derive
\begin{align}\label{eq:04222144}
&\liminf_{T \to\infty} \frac{1}{T}\meas\left\{ \tau \in [0,T] ~\middle|~ \sup_{s \in K} |\zeta(s+i \tau, \alpha)-f(s)|<\epsilon \right\} \\
&\geq \mathbf{P}\left(\|\zeta(s, \mathbb{X}_\alpha)-p(s)\|<\frac{\epsilon}{2C_K}\right) \nonumber
\end{align}
for any $\alpha \in \mathcal{A}$. 
Then, we prove that the last probability is positive if $\alpha \in \mathcal{A}_\rho(c) \setminus \mathcal{E}$, where $\rho$ is as in Corollary \ref{cor:4.9}, and $\mathcal{E}$ is as in Proposition \ref{prop:5.2}. 
As a consequence of Corollary \ref{cor:4.9}, there exists a sample $\omega_0 \in \Omega$ such that
\begin{gather*}
\|\zeta_N(s,\mathbb{Y}_\alpha)(\omega_0)-p(s)\|
< \frac{\epsilon}{8C_K}
\end{gather*}
for any $\alpha \in \mathcal{A}_\rho(c)$, where $N>N_0$ is an integer chosen later. 
Here, we can assume that $\rho$ satisfies $0<\rho \leq \min\{c, 1-c\}/2$. 
For $\alpha \in \mathcal{A}_\rho(c)$, we denote by $\Omega_0$ the event defined as \eqref{eq:04222013}. 
If $\omega \in \Omega_0$, then we have $\left|\mathbb{X}_\alpha(n)(\omega)-\mathbb{Y}_\alpha(n)(\omega_0)\right| \leq 2\pi \delta$ for any $0 \leq n \leq N$. 
It derives
\begin{align*}
\left\|\zeta_N(s,\mathbb{X}_\alpha)(\omega)- \zeta_N(s,\mathbb{Y}_\alpha)(\omega_0)\right\|
&\leq \sum_{n=0}^{N} \left|\mathbb{X}_\alpha(n)(\omega)-\mathbb{Y}_\alpha(n)(\omega_0)\right| \|(n+\alpha)^{-s}\| \\
&\leq \sum_{n=0}^{N} 2\pi \delta \sqrt{M} (n+c/2)^{-1/2} \\
&\leq \delta B_1 \sqrt{N}
\end{align*}
for $\omega \in \Omega_0$, where $B_1=B_1(c,K)$ is a positive constant depending only on $c$ and $K$. 
Assume that $\delta=\delta(N; \epsilon,K,B_1)$ satisfies
\begin{gather*}
0
< \delta
< \min \left\{\frac{1}{B_1 \sqrt{N}} \frac{\epsilon}{8C_K}, \frac{1}{2}\right\}. 
\end{gather*}
We can choose such a $\delta$ depending only on $c,\epsilon,K$, and $N$. 
Then, we see that the norm $\left\|\zeta_N(s,\mathbb{X}_\alpha)(\omega)- p(s)\right\|$ is evaluated as
\begin{align*}
&\left\|\zeta_N(s,\mathbb{X}_\alpha)(\omega)- p(s)\right\| \\
&\leq \left\|\zeta_N(s,\mathbb{X}_\alpha)(\omega)- \zeta_N(s,\mathbb{Y}_\alpha)(\omega_0)\right\|
+ \|\zeta_N(s,\mathbb{Y}_\alpha)(\omega_0)-p(s)\| \\
&< \frac{\epsilon}{4C_K}
\end{align*}
for any $\omega \in \Omega_0$. 
Hence we derive the inequality
\begin{align}\label{eq:04222058}
&\mathbf{P}\left(\|\zeta(s, \mathbb{X}_\alpha)-p(s)\|<\frac{\epsilon}{2C_K}\right) \\
&\geq \mathbf{P}\left(\Omega_0 \cap \left\{\|\zeta(s, \mathbb{X}_\alpha)-\zeta_N(s, \mathbb{X}_\alpha)\|<\frac{\epsilon}{4C_K}\right\}\right) \nonumber\\
&= \mathbf{P}(\Omega_0)
- \mathbf{P}\left(\Omega_0 \cap \left\{\|\zeta(s, \mathbb{X}_\alpha)-\zeta_N(s, \mathbb{X}_\alpha)\| \geq \frac{\epsilon}{4C_K}\right\}\right). \nonumber
\end{align}
By Chebyshev's inequality and Proposition \ref{prop:5.2}, we have 
\begin{align}\label{eq:04222059}
&\mathbf{P}\left(\Omega_0 \cap \left\{\|\zeta(s, \mathbb{X}_\alpha)-\zeta_N(s, \mathbb{X}_\alpha)\| \geq \frac{\epsilon}{4C_K}\right\}\right) \\
&\leq \left(\frac{4C_K}{\epsilon}\right)^2 
\mathbf{E}_{\Omega_0}\left[\|\zeta(s,\mathbb{X}_\alpha)-\zeta_N(s,\mathbb{X}_\alpha)\|^2\right] \nonumber\\
&\leq B_2 \left\{\mathbf{P}(\Omega_0) N^{1-2\sigma_0}
+ L^{1-2\sigma}
+ \frac{NL (\log{\Delta})^{N+1}}{\Delta}\right\} \nonumber
\end{align}
for any $\alpha \in \mathcal{A}_\rho(c) \setminus \mathcal{E}$, where $B_2=B_2(\epsilon,K)$ is a positive constant depending only on $\epsilon$ and $K$. 
Then it is deduced from \eqref{eq:04222058} and \eqref{eq:04222059} that
\begin{align}\label{eq:04222143}
&\mathbf{P}\left(\|\zeta(s, \mathbb{X}_\alpha)-p(s)\|<\frac{\epsilon}{2C_K}\right) \\
&\geq (1-B_2N^{1-2\sigma_0}) \mathbf{P}(\Omega_0)
- B_2 L^{1-2\sigma_0}
- B_2 \frac{NL (\log{\Delta})^{N+1}}{\Delta}. \nonumber
\end{align}
We choose the integer $N=N(\sigma_0,N_0,B_2)>N_0$ so that $B_2 N^{1-2\sigma_0}\leq1/2$ is satisfied. 
Note that $\sigma_0$ depends only on $K$, and that $N_0$ of Corollary \ref{cor:4.9} depends only on $c,\epsilon,f,K$. 
Thus the integer $N$ depends only on $c,\epsilon,f,K$. 
Then we obtain
\begin{gather*}
(1-B_2N^{1-2\sigma_0}) \mathbf{P}(\Omega_0)
\geq \frac{1}{2}\mathbf{P}(\Omega_0)
\geq \frac{1}{2}\delta^{N+1}
- B_3 \frac{N(\log{\Delta})^{N+1}}{\Delta}
\end{gather*}
for any $\alpha \in \mathcal{A}_\rho(c) \setminus \mathcal{E}$ by Proposition \ref{prop:5.2}, where $B_3=B_3(K)$ is a positive constant depending only on $K$. 
Furthermore, we choose $L=L(\delta,\sigma_0,N,B_2)>N$ so that 
\begin{gather*}
B_2 L^{1-2\sigma_0}
< \frac{1}{4} \delta^{N+1} 
\end{gather*}
is satisfied. 
Lastly, we choose a real number $\Delta=\Delta(\delta,N,L,B_2,B_3)>\Delta_0$ so that
\begin{gather*}
B_2 \frac{NL (\log{\Delta})^{N+1}}{\Delta}
+ B_3 \frac{N(\log{\Delta})^{N+1}}{\Delta}
< \frac{1}{4} \delta^{N+1} 
\end{gather*}
is satisfied. 
As a result, we derive by \eqref{eq:04222143} that 
\begin{gather}\label{eq:04222145}
\mathbf{P}\left(\|\zeta(s, \mathbb{X}_\alpha)-p(s)\|<\frac{\epsilon}{2C_K}\right)
> 0
\end{gather}
for any $\alpha \in \mathcal{A}_\rho(c) \setminus \mathcal{E}$. 
From the above setting, $N$, $L$, and $\Delta$ depend only on $c,\epsilon,f,K$, and therefore, the finite set $\mathcal{E}$ is determined only by $c,\epsilon,f,K$. 
By \eqref{eq:04222144} and \eqref{eq:04222145}, we obtain the desired result. 
\end{proof}

\section{Consequences of the main result}\label{sec:6}

Let $f \in H(D)$ and $1/2<\sigma_0<1$. 
By Cauchy's integral formula, we have 
\begin{gather}\label{eq:04090100}
f^{(n)}(\sigma_0)
= \frac{n!}{2\pi i} \oint_{|s-\sigma_0|=r} \frac{f(s)}{(s-\sigma_0)^{n+1}} \,ds 
\end{gather}
for $n \geq0$, where $r$ satisfies $0<r<\min\{\sigma_0-1/2, 1-\sigma_0\}$. 
Applying this formula, we prove that Theorem \ref{thm:1.4} implies Theorem \ref{thm:1.5}. 

\begin{proof}[Proof of Theorem \ref{thm:1.5}]
Denote by $K$ the disc $K= \{s \in \mathbb{C} \mid |s-\sigma_0|\leq r \}$, where $r$ is taken as
$r=\min\{\sigma_0-1/2, 1-\sigma_0\}/2$. 
Then $K$ is a compact subset of the strip $D$ with connected complement. 
Define the function $f$ as 
\begin{gather*}
f(s)
= \sum_{n=0}^{N} \frac{z_n}{n!} (s-\sigma_0)^n
\end{gather*}
by using $\underline{z}=(z_0, \ldots, z_N) \in \mathbb{C}^{N+1}$. 
Note that $f^{(n)}(\sigma_0)=z_n$ by definition. 
Therefore, we deduce from \eqref{eq:04090100} that 
\begin{align*}
|\zeta^{(n)}(\sigma_0+i \tau, \alpha)-z_n|
&= \left| \frac{n!}{2\pi i} \oint_{|s-\sigma_0|=r} \frac{\zeta(s+i \tau, \alpha)-f(s)}{(s-\sigma_0)^{n+1}} \,ds \right| \\
&\leq \frac{n!}{r^n} \sup_{s \in K} |\zeta(s+i \tau, \alpha)-f(s)|. 
\end{align*}
For every $\epsilon>0$, we put
\begin{gather*}
\epsilon'
= \epsilon \cdot \left(1+\frac{1}{r}+\cdots+\frac{N!}{r^N}\right)^{-1}. 
\end{gather*}
Then we see that $|\zeta^{(n)}(\sigma_0+i \tau, \alpha)-z_n|<\epsilon$ is satisfied for any $0 \leq n \leq N$ if we suppose $\sup_{s \in K} |\zeta(s+i \tau, \alpha)-f(s)|<\epsilon'$. 
Since the function $f$ is continuous on $K$ and analytic in the interior of $K$, we can apply Theorem \ref{thm:1.4}. 
Hence, there exist a positive real number $\rho$ and a finite subset $\mathcal{E} \subset \mathcal{A}_\rho(c)$ depending on $c,\epsilon',f,K$ such that 
\begin{align*}
&\liminf_{T \to\infty} \frac{1}{T} \meas
\left\{ \tau \in [0,T] ~\middle|~ \max_{0 \leq n \leq N} |\zeta^{(n)}(\sigma_0+i \tau, \alpha)-z_n|<\epsilon \right\} \\
&\geq \liminf_{T \to\infty} \frac{1}{T} \meas
\left\{ \tau \in [0,T] ~\middle|~ \sup_{s \in K} |\zeta(s+i \tau, \alpha)-f(s)|<\epsilon' \right\} 
> 0
\end{align*}
for any $\alpha \in \mathcal{A}_\rho(c) \setminus \mathcal{E}$. 
Recall that $\epsilon', f, K$ are determined only by $\epsilon, \sigma_0, \underline{z}$. 
Thus $\rho$ and $\mathcal{E}$ depend on $c,\epsilon,\sigma_0, \underline{z}$, and we obtain the conclusion. 
\end{proof}

For the proof of Theorem \ref{thm:1.6}, we apply Rouche's theorem \cite[Theorem 10.43]{Rudin1987}. 
Let $f,g \in H(D)$ and $1/2<\sigma_0<1$. 
Suppose that the inequality
\begin{gather}\label{eq:04090148}
\sup_{|s-\sigma_0|=r} |g(s)-f(s)|
< \inf_{|s-\sigma_0|=r} |f(s)|
\end{gather}
holds, where $r$ satisfies $0<r<\min\{\sigma_0-1/2, 1-\sigma_0\}$. 
Then $g$ has the same number of zeros as that of $f$ in the region $|s-\sigma_0|<r$. 

\begin{proof}[Proof of Theorem \ref{thm:1.6}]
For $1/2<\sigma_1<\sigma_2<1$, we take the real numbers $\sigma_0$ and $r$ as $\sigma_0=(\sigma_1+\sigma_2)/2$ and $r=(\sigma_2-\sigma_1)/4$. 
Then $K= \{s \in \mathbb{C} \mid |s-\sigma_0|\leq r \}$ is a compact subset of the strip $D$ with connected complement. 
Furthermore, we define the function $f$ as 
\begin{gather*}
f(s)
= s-\sigma_0. 
\end{gather*}
Obviously, it is continuous on $K$ and analytic in the interior of $K$. 
Therefore we can apply Theorem \ref{thm:1.4}. 
Then there exist a positive real number $\rho$ and a finite subset $\mathcal{E} \subset \mathcal{A}_\rho(c)$ such that 
\begin{gather}\label{eq:04090141}
\liminf_{T \to\infty} \frac{1}{T} \meas
\left\{ \tau \in [0,T] ~\middle|~ \sup_{s \in K} |\zeta(s+i \tau, \alpha)-f(s)|<r \right\} 
> 0
\end{gather}
for any $\alpha \in \mathcal{A}_\rho(c) \setminus \mathcal{E}$. 
Define $\mathfrak{S}(\alpha)$ as the set of all $\tau \in \mathbb{R}_{\geq0}$ such that the inequality $\sup_{s \in K} |\zeta(s+i \tau, \alpha)-f(s)|<r$ is satisfied. 
Then, for any $\alpha \in \mathcal{A}_\rho(c) \setminus \mathcal{E}$, there exists a sequence $\{\tau_n\}$ of elements in $\mathfrak{S}(\alpha)$ such that 
\begin{gather*}
\tau_{n+1}
\geq \tau_n+2r
\quad\text{and}\quad
\mathfrak{S}(\alpha)
\subset \bigcup_{n=1}^{\infty} [\tau_n-r, \tau_n+r] 
\end{gather*}
by \eqref{eq:04090141}. 
Furthermore, we see that
\begin{gather*}
\sup_{|s-\sigma_0|=r} |\zeta(s+i \tau_n, \alpha)-f(s)|
< r
= \inf_{|s-\sigma_0|=r} |f(s)|
\end{gather*}
for such $\tau_n$. 
Hence \eqref{eq:04090148} holds with $g(s)=\zeta(s+i \tau_n, \alpha)$. 
By Rouche's theorem, the function $\zeta(s+i \tau_n, \alpha)$ has exactly one zero in $|s-\sigma_0|<r$, that is, $\zeta(s, \alpha)$ has exactly one zero in the region 
\begin{gather*}
U_n
= \{s \in \mathbb{C} \mid |s-(\sigma_0+i \tau_n)|<r \} 
\end{gather*}
for any $n \geq1$. 
Note that the regions $U_n$ are distinct by $\tau_{n+1}\geq \tau_n+2r$. 
As a result, we obtain
\begin{gather}\label{eq:04090417}
N_\alpha(\sigma_1,\sigma_2,T)
\geq n(T)
:=\max \left\{n ~\middle|~ \tau_n+r \leq T \right\} 
\end{gather}
since $\bigcup_{n=1}^{n(T)} U_n$ is included in the rectangle $\sigma_1 \leq \sigma \leq \sigma_2$, $0 \leq t \leq T$. 
On the other hand, the inequality
\begin{gather}\label{eq:04090418}
\meas (\mathfrak{S}(\alpha) \cap [0,T])
\leq \meas \left(\bigcup_{n=1}^{n(T)+1} [\tau_n-r, \tau_n+r]\right)
= 2r \left(n(T)+1\right) 
\end{gather}
holds since $\mathfrak{S}(\alpha) \cap [0,T]$ is covered by $\bigcup_{n=1}^{n(T)+1} [\tau_n-r, \tau_n+r]$. 
Combining \eqref{eq:04090141}, \eqref{eq:04090417}, and \eqref{eq:04090418}, we have
\begin{gather*}
N_\alpha(\sigma_1,\sigma_2,T)
\gg \meas (\mathfrak{S}(\alpha) \cap [0,T]) 
\gg T
\end{gather*}
as $T \to\infty$ for any $\alpha \in \mathcal{A}_\rho(c) \setminus \mathcal{E}$. 
Here, we recall that $\rho$ and $\mathcal{E}$ depend on $c,r,f,K$. 
Since $r,f,K$ are determined only by $\sigma_1$ and $\sigma_2$, they depend on $c,\sigma_1,\sigma_2$. 
Therefore the proof is completed.  
\end{proof}


\providecommand{\bysame}{\leavevmode\hbox to3em{\hrulefill}\thinspace}
\providecommand{\MR}{\relax\ifhmode\unskip\space\fi MR }
\providecommand{\MRhref}[2]{%
  \href{http://www.ams.org/mathscinet-getitem?mr=#1}{#2}
}
\providecommand{\href}[2]{#2}

\end{document}